\documentclass{elsarticle}
\usepackage[utf8]{inputenc}
\usepackage[english]{babel}
\usepackage{fullpage}
\usepackage{amsmath, amsfonts, amssymb, amsthm, bm, xcolor, graphicx, hyperref}
\usepackage{capt-of}
\usepackage{changes}

\newcommand{\bea}{\begin{eqnarray}}
\newcommand{\eea}{\end{eqnarray}}
\newcommand{\beas}{\begin{eqnarray*}}
\newcommand{\eeas}{\end{eqnarray*}}
\newcommand{\be}{\begin{equation}}
\newcommand{\ee}{\end{equation}}
\newcommand{\bfR}{\mathbb{R}}

\newcommand{\Cc}{\mathbb{C}}
\newcommand{\colspan}{\mathrm{colspan}}

\newcommand{\jend}{j_{\textsf{end}}}
\newcommand{\jstart}{j_{\textsf{start}}}

\newcommand{\Rr}{\mathbb{R}}
\newcommand{\tol}{\texttt{tol}}

\newtheorem{proposition}{Proposition}[section]

\newtheorem{definition}{Definition}[section]

\newtheorem{remark}[definition]{Remark}
\renewcommand{\le}{\leqslant}
\renewcommand{\ge}{\geqslant}

\begin{document}

\begin{frontmatter}

\title{A single LU~factorization approach to model reduction\\for linear
system families with a quadratic matrix pencil}


\author[2]{Mikhail A. Botchev\corref{cor1}}
\ead{botchev@kiam.ru}

\author[1]{Leonid Knizhnerman}
\ead{lknizhnerman@gmail.com}

\author[1]{Nikolai L. Zamarashkin}
\ead{nikolai.zamarashkin@gmail.com}

\author[1]{Dmitry A. Zheltkov}
\ead{dmitry.zheltkov@gmail.com}

\affiliation[2]{organization={Keldysh Institute of Applied Mathematics of Russian Academy of Sciences},
addressline={Miusskaya Sq.~4},
city={Moscow},
postcode={125047},
country={Russia}}

\affiliation[1]{organization={Marchuk Institute of Numerical Mathematics of Russian Academy of Sciences},
addressline={Gubkina St.~8},
city={Moscow},
postcode={119333},
country={Russia}}

\cortext[cor1]{Corresponding author}

\begin{abstract}
To solve a family of large linear systems where
the system matrix depends on the family parameter quadratically, 
we propose a new solution approach.  
Our method employs the derivatives of solutions with respect to the parameter.
%
%
%
It turns out that the accumulated solution derivatives of increasing order, 
evaluated at a single specially chosen parameter value, 
provide successively better approximations to solution subspaces 
for the whole linear system family.

We show that the derivatives can be computed by a recurrence,
involving, at each step, solving a linear system with the same matrix.
Constructing an orthonormal basis for the subspace spanned by the derivatives 
then leads to an efficient solver where 
just a single (sparse) LU factorization 
(or a single preconditioner setup in an iterative scheme) 
suffices to solve the whole linear system family.
We call the new method ODB (orthonormal derivative basis method).
As shown in numerical tests, compared to conventional solution methods, 
the ODB method allows to drastically reduce computational costs.
\end{abstract}


\end{frontmatter}


\section{Introduction}
In this paper we propose a method to solve a family of linear systems
\begin{equation}
\label{mor}
(A_0 + kA_1 + k^2 A_2) \bm{x} = k\bm{b},
\end{equation}
where parameter values $k>0$, 
$N\times N$ matrices $A_0$, $A_1$, $A_2$ and the right hand side vector
$\bm{b}\in\Cc^N$ are given.
Linear system families~\eqref{mor} arise in different applications and, in particular,
in numerical solution of the Maxwell equations.  
Although particular properties of the matrices $A_j$ may vary, 
depending on application, in the context of discretized PDEs typically
$A_0$ is a stiffness matrix, $A_1$ contains contributions
of the boundary terms, and $-A_2$ is a mass matrix.
More specifically, for electromagnetic simulation problems considered in this paper 
and described below in Section~\ref{s:problem}, we have 
\begin{equation}
\label{mor_prop}
\begin{gathered}
\Rr^{N\times N}\ni A_0=A_0^T\ge 0 , 
\qquad 
A_1=iS, \quad \Rr^{N\times N}\ni S=S^T\ge 0, 
\\
A_2 = -M + i M_{\delta}, 
\quad M,M_\delta\in\Rr^{N\times N}, 
\quad M=M^T > 0, 
\quad M_\delta=M_\delta^T \geqslant 0,
\end{gathered}
\end{equation}
with $X=X^T> 0$ and $X=X^T\ge 0$ meaning throughout this paper that a matrix $X$ 
is symmetric positive definite and, respectively, semidefinite,
and $i$ being the imaginary unit ($i^2 = -1$).
The systems in~\eqref{mor} have to be solved for 
$J\geqslant 2$ given positive values of the parameter~$k$, i.e., for
\begin{equation}
\label{mor_k}
k\in\{k_1,k_2,\dots,k_J\},\quad\text{with}\quad
0<k_{\min}=k_1<k_2<\dots<k_J=k_{\max}.
\end{equation}
In typical settings the system size~$N$ can be very large, having 
$N\sim 10^8$ is not unusual. In addition, the linear systems in~\eqref{mor} 
may have to be solved not only for different values of the parameter~$k$ but also 
for up to hundreds various right hand side vectors~$\bm{b}$.

Due to the large system size and necessity to solve systems~\eqref{mor}
for many various $k$ parameter values and right hand sides,
a usual way to efficiently handle the linear system family in~\eqref{mor}
is to apply well developed model-order reduction (MOR) 
techniques, see, e.g.,~\cite{Schilders_ea2008book,QuarteroniRozza2014book}.
Last decades the MOR methods have formed an actively developing research
area.
We mention classical MOR solution approaches, such as
Principal Component Analysis (PCA) and Proper Orthogonal Decomposition (POD)
methods initiated in~\cite{Moore1981,Sirovich1987}
and Krylov subspace methods, including the PVL (Pad\'e via Lanczos) family 
of methods~\cite{FeldmannFreund1995,Freund2003}.
Later a number of other efficient MOR approaches have been developed.
These include 
PRIMA methods~\cite{Odabasioglu_ea1998} 
(which can be seen as block Arnoldi process generalization of PVL)
and more advanced algorithms of this type~\cite{Freund2004,Freund2011},
the POD-DEIM group of 
methods~\cite{ChaturantabutSorensen2010_DEIM,ChaturantabutSorensen2012,Petrov2020} and recently proposed
symplectic MOR methods~\cite{Hesthaven_ea2021}.

In this paper an alternative MOR approach for solving~\eqref{mor} is proposed 
which allows to decrease computational costs drastically.
The key issue in the MOR methods is, apparently, how to find a proper 
approximate solution
subspace which would, on one hand, have the smallest possible dimension and, 
on the other hand,
reflect all relevant solution properties of the linear system family being solved.
In our experience (in particular, see numerical results
in Section~\ref{s:exp}), for the problems we aim at, the dimension of appropriate 
MOR solution subspaces in a conventional MOR method is roughly 
proportional to computational costs.
More precisely, to find an appropriate solution subspace of the smallest 
possible dimension~$m$, conventional MOR methods need to solve 
the linear system~\eqref{mor} for about $m$ different parameter
values~$k$.

Our solution approach proposed here is free from this restriction.
It is based on the observation that solution $\bm{x}$ is smooth
with respect to the parameter~$k$.
Indeed, as 
we usually observe in real-life
applications, 
mapping $\bm{x}(k)$ is, at least locally, analytic.
The analyticity assumption allows to propose an efficient solution algorithm which,
to find the MOR solution subspace at low cost,
employs the derivatives 
$\bm{x}(\hat{k})$, $\bm{x}'(\hat{k})$, $\bm{x}''(\hat{k})$, \dots{}
evaluated at some properly chosen value~$\hat{k}$.
This point of view appears to give us two major advantages.
Namely, the first key feature of our approach is that 
all the derivatives $\bm{x}^{(m)}(\hat{k})$
can be computed by a recurrence which involves solving linear systems with 
a single matrix $A_0 + \hat{k}A_1 + \hat{k}^2 A_2$ only, 
which is evaluated at the chosen point~$\hat{k}$.
The second key advantage with our approach is that there is a freedom 
in choosing this specific point~$\hat{k}$.  
In fact, it turns out that a MOR subspace better approximating all the solutions 
in~\eqref{mor} can be built for a~$\hat{k}$ taken outside the interest range 
$[k_{\min},k_{\max}]$, namely, for a complex~$\hat{k}$.  
As we show, taking a complex~$\hat{k}$ is crucial as it allows to have 
a single derivative recursion and, hence, just a single (sparse) LU factorization,
to solve the whole family system~\eqref{mor} for all $k\in[k_{\min},k_{\max}]$.  

Moreover, taking a complex~$\hat{k}$ not only yields better approximation
properties of the MOR basis being constructed but also leads to another
important improvement.  As will be clear from the presentation below,
at least in the application class we consider here, 
the matrices $A_0 + kA_1 + k^2 A_2$ for $k\in[k_{\min},k_{\max}]$ are moderately
ill-conditioned, so that sparse direct solvers can, in principle, 
be successfully employed to solve the systems in~\eqref{mor}.
At the same time, employing preconditioned Krylov subspace iterative solvers 
to solve 
systems~\eqref{mor} appears to be practically impossible due to 
extremely severe spectral properties of the matrices 
$A_0 + kA_1 + k^2 A_2$ for $k\in[k_{\min},k_{\max}]$.
It turns out taking the complex value~$\hat{k}$ (well suitable for the MOR 
basis approximation) also, as a side effect,
dramatically improves spectral properties of the pencil 
$A_0 + \hat{k}A_1 + \hat{k}^2 A_2$, so that, if necessary, 
preconditioned Krylov subspace iterative solvers can successfully be 
employed to solve the systems with the pencil.

If a direct method is used to solve
the linear system at point~$\hat{k}$ then, as appears in all the tests
we have carried out so far, just a single LU factorization suffices
to compute the whole MOR solution basis.  Likewise, in case an iterative method is 
employed for solving the linear system at~$\hat{k}$, just 
a single preconditioner setup is enough to compute the MOR solution basis.
An important ingredient of our approach is that, to avoid eventual
linear dependence of the derivatives $\bm{x}^{(m)}(\hat{k})$
we apply a special orthogonalization procedure.  This is rather similar to 
the Gram-Schmidt orthogonalization applied to the Krylov basis vectors 
in classical Arnoldi (Lanczos) methods.
We call our method \emph{orthonormal derivative basis} (ODB) method.

The structure of this paper is as follows. 
A typical problem setting leading to linear system family~\eqref{mor} 
is given in Section~\ref{s:problem}.
Some possible conventional MOR solution approaches suitable 
to solve systems~\eqref{mor} and 
a regular MOR solver used for comparisons with the new solver
are discussed in Section~\ref{s:POD}.
Our ODB method is presented in Section~\ref{s:ODB}. 
There we present the solution concept using the solution 
derivatives $\bm{x}^{(m)}(\hat{k})$ (Section~\ref{s:ODBder}),
propose an algorithm to recursively construct an orthonormal
basis of 
$\text{span} (\bm{x}(\hat{k}), \bm{x}'(\hat{k}), \dots, \bm{x}^{(m)}(\hat{k}))$
(Section~\ref{s:ODBortho})
and explain how to choose the point~$\hat{k}$ (Section~\ref{s:anal}) properly.
Numerical tests are discussed in Section~\ref{s:exp}
and conclusions are drawn in Section~\ref{s:concl}.

Throughout the paper $I$ denotes the identity matrix of a suitable
dimension and, unless indicated otherwise, $\|\cdot\|$ is the Euclidean vector norm 
(the 2-norm) or the corresponding
operator (matrix) norm.

\medskip

\noindent
\textbf{Highlights}

\begin{enumerate}
\item 
A new ODB (orthonormal derivative basis) MOR approach is proposed 
for solving linear system families, with the system matrix being
a quadratic matrix pencil of a parameter.
The approach allows to drasctically decrease computational costs:
a single LU~factorization suffices to solve a linear system family 
for the whole parameter range.

\item
Our approach is based on two key ideas:\\
(a)~generate an orthonormal MOR solution basis spanned by the derivatives
$\bm{x}(k)$, $\bm{x}'(k)$, $\bm{x}''(k)$, \dots{} of the solution $\bm{x}(k)$
with respect to the parameter~$k$;\\
(b)~this orthonormal derivative basis is computed at a certain complex 
parameter value $\hat{k}$ taken outside the real interest range
$k\in[k_{\min},k_{\max}]$, at a price of just a single LU~factorization.

\item
We show how to choose a single $\hat{k}$ parameter value, where the basis 
is computed, such that it yields a good approximation to solutions $\bm{x}(k)$
for the whole interest range $k\in[k_{\min},k_{\max}]$.
Moreover, for the properly chosen~$\hat{k}$, the linear systems can be efficiently 
solved by the Krylov subspace iterative methods, while this is not the case
for $k\in[k_{\min},k_{\max}]$.
\end{enumerate}

\section{Problem formulation}
\label{s:problem}

One typical application, where linear system families~\eqref{mor} arise,
is electromagnetic scattering in antenna design.   
In this case a usual starting point for deriving~\eqref{mor} is the vector 
Helmholtz equation for the electric field $\bm{E}$,
for which a weak formulation can be given as follows:
\begin{equation}
\label{WeakForm}
\begin{aligned}
\iiint_\Omega & \left[(\nabla\times \bm{v})\cdot(\nabla\times\bm{E})
-k^2\mu\varepsilon(\bm{x})(1-i\tan\delta(\bm{x}))\bm{v}\cdot\bm{E}
\right]\,d\Omega 
\\
&+ ik\mu\sqrt{\frac{\mu_0}{\varepsilon_0}} \sum_{j=1}^{N_p}
\iint_{\Gamma_{\mbox{\scriptsize port}_j}} \frac{1}{Z_{s,j}}
(\bm{n}\times\bm{v})\cdot(\bm{n}\times\bm{E})\,d\Gamma
+ ik\iint_{\Gamma}
(\bm{n}\times\bm{v})\cdot(\bm{n}\times\bm{E})\,d\Gamma 
\\
& \hspace*{15em} = \frac{2ik}{Z_{s,m}}\,\mu\sqrt{\frac{\mu_0}{\varepsilon_0}}
\iint_{\Gamma_{\mbox{\scriptsize port}_m}}
(\bm{n}\times\bm{v})\cdot(\bm{n}\times\bm{E}_0)\,d\Gamma,
\quad\forall\bm{v},
\end{aligned}
\end{equation}
where: $\Omega\subset\bfR^3$ is a bounded computational domain; $\Gamma$ 
is its boundary;
$\bm{E}$ is the unknown complex-valued vector electric field; 
$\bm{v}$ is a function from the test space;
$\Gamma_{\mbox{\scriptsize port}_j}$, $j=1,\ldots,N_p$, are port boundaries 
(with $N_p$ being the number of ports);
$\bm{n}$ is the external normal unit vector to the corresponding boundary;
$i$ is the imaginary unit;
$k=\omega/c$ is the wave number in the vacuum 
(with $\omega$ being the angular frequency of the wave and
$c$ the speed of light in the vacuum); 
$\mu$ is the magnetic permeability 
($\mu_0$ is magnetic permeability in vacuum);
$\varepsilon$ is the dielectric permittivity
($\varepsilon_0$ is dielectric permittivity in vacuum);
$\tan\delta(\bm{x})$ is the dielectric loss tangent;
$Z_{s,j}$ is the sheet impedance of port number~$j$;
$\bm{E}_0$ is given electric source field;
and, finally, $m$ is the number of the excited port.

The linear system family~\eqref{mor} is obtained after
a straightforward discretization of~\eqref{WeakForm}, 
typically done by the vector N\'ed\'elec finite elements~\cite{Nedelec80,Nedelec86}.
Applying the finite element discretization and dividing all the terms in~\eqref{WeakForm}
into three groups, containing respectively zeroth, first and second powers of~$k$,
we arrive at~\eqref{mor}. 
Then, we obtain
\begin{gather*}
A_0\bm{e}\approx 
\iiint_\Omega  \left[(\nabla\times \bm{v})\cdot(\nabla\times\bm{E})\right]\,d\Omega,
\\
A_1\bm{e}\approx 
i\mu\sqrt{\frac{\mu_0}{\varepsilon_0}} \sum_{j=1}^{N_p}
\iint_{\Gamma_{\mbox{\scriptsize port}_j}} \frac{1}{Z_{s,j}}
(\bm{n}\times\bm{v})\cdot(\bm{n}\times\bm{E})\,d\Gamma
+ i\iint_{\Gamma}
(\bm{n}\times\bm{v})\cdot(\bm{n}\times\bm{E})\,d\Gamma 
\\  
A_2\bm{e}\approx 
- \iiint_\Omega  \left[\mu\varepsilon(\bm{x})(1-i\tan\delta(\bm{x}))\bm{v}\cdot\bm{E}
\right]\,d\Omega,
\end{gather*}
where $\bm{e}$ is a vector containing the finite element degrees of freedom
for the electric field~$\bm{E}$.  Thus,
the matrix $A_0$ is a discrete curl-curl operator,
$A_1$ contains all the contributions of the boundary integrals
from the second line in~\eqref{WeakForm},
and $A_2$ is the mass matrix, with real and imaginary parts given by
the matrices $-M$ and $M_\delta$, respectively. 
It is then straightforward to check that $A_{0,1,2}$ satisfy properties~\eqref{mor_prop}.


\section{Conventional MOR methods}
\label{s:POD}

\subsection{Possible regular MOR solution approaches}
Let $\bm{x}_j=\bm{x}(k_j)$ be the solution of~\eqref{mor} for $k=k_j$.
A key assumption, which is characteristic for MOR methods and is also made here, 
is that all the solution vectors $\bm{x}_j$ can be sufficienly well
approximated by elements of a low-dimensional subspace, 
with dimension~$m\ll N$, i.e.,
\begin{equation}
\label{x=Vy}
\bm{x}_j \approx y_1^{(j)}\bm{v}_1 + y_2^{(j)}\bm{v}_2 + \dots +
                 y_m^{(j)}\bm{v}_m = 
V y_j, \quad  j=1,\dots,J,
\end{equation}
where the vectors $\bm{v}_1$, \dots, $\bm{v}_m$ are linearly independent 
columns of $V\in\Cc^{N\times m}$ and $y_i^{(j)}$ are the entries of the 
vector $y_j\in\Cc^m$.
Since it is practical to work with an orthonormal basis,
we assume, without loss of generality, that
$$
V^*V = I\in\Rr^{m\times m}.
$$
To get a working solution algorithm, one can start, for some $V\in\Cc^{N\times m}$, 
with an approximate solution $V y_j\approx\bm{x}_j$ in~\eqref{x=Vy} and extend 
the matrix $V$ by successively adding columns (and thus increasing $m$).
Replacing~$\bm{x}_j$ in the systems~\eqref{mor} by an approximation 
$V y_j\approx\bm{x}_j$ gives to a family of projected overdetermined 
linear systems in~$y=y_j$
\begin{equation}
\label{mor_appr}
(A_0V + kA_1V + k^2 A_2V) y \approx k\bm{b},
\qquad k\in\{k_1,k_2,\dots,k_J\},
\end{equation}
where, for each $k=k_j$, a suitable $y$ has to be chosen.  
This can be done in several natural ways.

First, we can impose the Galerkin condition
\begin{equation}
\label{mor_Gal}
V^*(A_0V + kA_1V + k^2 A_2V) y = V^*(k\bm{b}),
\qquad k\in\{k_1,k_2,\dots,k_J\},
\end{equation}
which means that, for each $k=k_j$, the residuals 
\begin{equation}
\label{resid}  
\bm{r} = k\bm{b} - (A_0V + kA_1V + k^2 A_2V) y   
\end{equation}
are orthogonal to the search subspace $\colspan(V)$, the column space of $V$,
i.e., $V^* \bm{r} = 0$.
Note that the linear systems in~\eqref{mor_Gal} have $m$ equations
and $m$ unknowns.

Another well known approach to choose a solution $y$ in~\eqref{mor_appr} 
is to minimize the residual norms:
\begin{equation}
\label{mor_mr}
\| \bm{r} \| = \| k\bm{b} - (A_0V + kA_1V + k^2 A_2V) y \| 
\rightarrow \min_{y\in\Cc^{{m}}},   
\qquad k\in\{k_1,k_2,\dots,k_J\}.
\end{equation}
In this case, denoting 
$\widehat{y}=\begin{bmatrix} y^T & ky^T & k^2y^T \end{bmatrix}^T\in\Cc^{3{m}}$,
we can write
\begin{equation}
\label{res}  
\bm{r} = k\bm{b} - 
\begin{bmatrix}
A_0V & A_1V & A_2V
\end{bmatrix}\widehat{y}
\end{equation}
and orthonormalize the columns of the matrix 
$\begin{bmatrix} A_0V & A_1V & A_2V \end{bmatrix}$,
thus computing its thin QR factorization
$$
\begin{bmatrix} A_0V & A_1V & A_2V \end{bmatrix} = Q
\begin{bmatrix}R_0 & R_1 & R_2 \end{bmatrix},
\qquad Q^*Q=I.
$$ 
Here $Q\in\Cc^{N\times 3{m}}$, $R_{0,1,2}\in\Cc^{3{m}\times {m}}$,
and the matrix $\begin{bmatrix}R_0 & R_1 & R_2 \end{bmatrix}$ is upper triangular.
Let
\begin{equation}
\label{d}  
d = \arg\min_{d\in\Cc^{3{m}}}\|\bm{b} - Qd\|.
\end{equation}
It is natural to expect (and we observe this in actual computations)
that already at the first steps of constructing the test space 
$\colspan(V)$  
the right hand side vector $\bm{b}$ is in the column space of 
the matrix $\begin{bmatrix} A_0V & A_1V & A_2V \end{bmatrix}$, i.e., 
$$
\bm{b}\in \colspan(\begin{bmatrix} A_0V & A_1V & A_2V \end{bmatrix}) = \colspan(Q).
$$
Then we have $\bm{b}=Qd$ and the minimum in~\eqref{d} is zero.
In any case, for $d$ defined in~\eqref{d}, we have
$$
\begin{aligned}
\|\bm{r}\| &= \left\|Qkd - Q\begin{bmatrix}
R_0 & R_1 & R_2
\end{bmatrix}\widehat{y})\right\| = 
\left\|kd - \begin{bmatrix}
R_0 & R_1 & R_2
\end{bmatrix}\widehat{y}\right\| =
\\
&= \left\|kd - (R_0 + kR_1 + k^2R_2)y\right\|.
\end{aligned}
$$
Thus, the problem of minimizing the residual norm in~$\Cc^N$
is reduced to minimizing the norm of the projected residual
$d - (R_0 + kR_1 + k^2R_2)y \in\Cc^{3{m}}$.  This last problem
can be solved by the regular least squares method.

A third possible way of approximating~\eqref{mor_appr} is, 
after solving systems~\eqref{mor} for some selected values of $k_{\tilde{j}}$, 
to interpolate the vectors $y_j$ from the neighboring vectors $y_{\tilde{j}}$ obtained for the ``solved'' values.    

These three ways to choose~$y$ in~\eqref{mor_appr}, i.e.,
Galerkin~\eqref{mor_Gal}, minimal residual~\eqref{mor_mr} and 
interpolation approaches, show how to \emph{use} the approximate
solution subspace spanned by the columns of~$V$.  
To \emph{extend} the search subspace, we can proceed as follows.
We check, with the current $V$, the residual norms $\|\bm{r}\|$
in~\eqref{resid} for each $k=k_j$, choose $j$ for which the residual
norm is the largest, solve the system~\eqref{mor} for this $j$ and add 
the computed solution to the columns of $V$.
This approach, combined with the Galerkin projection~\eqref{mor_Gal}
and interpolation of the ``unsolved'' vectors $y_j$,
appears to work successfully for the application class we aim at,
and therefore we use this solution approach as a regular MOR solver.

\subsection{The regular MOR solver and its limitations}
An algorithm for the just outlined regular MOR solver, solving linear system 
family~\eqref{mor} for a given relative accuracy tolerance~$\tol$, 
is outlined in Figure~\ref{f:base}.

\begin{figure}
\hrule
\begin{enumerate}
\item 
Find $\bm{x}_1$ and $\bm{x}_J$ by solving~\eqref{mor} for $k_1$ and $k_J$, 
orthonormalize $\bm{x}_1$, $\bm{x}_J$ (by the modified Gram--Schmidt process) 
and set them as columns of $V\in\Cc^{N\times 2}$. Mark $k_1$ and $k_J$ as ``solved''.

\item
Form $W_{0,1,2}:=A_{0,1,2}V$, $H_{0,1,2}=V^*W_{0,1,2}$

\item
For each not ``solved'' $k_j$:
\begin{enumerate}
\item
find $\hat{y}_j$ from 
$(H_0 + k_j H_1 + k_j^2 H_2) \hat{y}_j = V^*(k_j\bm{b})$
\\ 
find $\tilde{y}_j$ by linearly interpolating all $y_i$ with solved $k_i$

\item
for $y_j=\hat{y}_j$ and $y_j=\tilde{y}_j$ 
compute the residual $\bm{r}_j$ of $\bm{x}_j = Vy_j$ as\\  
$\bm{r}_j = k_j\bm{b} - W_0y_j - k W_1y_j - k^2W_2y_j$\\
set either $y_j:=\hat{y}_j$ or $y_j:=\tilde{y}_j$ so that 
$\|\bm{r}_j\|$ is smallest

\item
mark $k_j$ as ``solved'' if $\|\bm{r}_j\|/ \|k_j\bm{b}\| \leqslant \tol$
\end{enumerate}

\item
If all $k_j$ are marked as ``solved'' then STOP

\item
Solve $(A_0+kA_1 + k^2A_2)\bm{x}_j = k_j\bm{b}$ for $j$ with 
the largest $\|\bm{r}_j\|/\|k_j\bm{b}\|$, 
add $\bm{x}_j$ as a column to $V$ and 
apply the Gram-Schmidt process to have $V^*V=I$. Go to step~2.
\end{enumerate}
\hrule

\caption{The regular MOR solver
to solve the system family~\eqref{mor} 
for a given relative accuracy tolerance~$\tol$.
Note that the vectors $\hat{y}_j$, defined at step 3(a) of the algorithm, 
should not be confused with $\widehat{y}$ introduced in~\eqref{res}.
Note that we could, in principle, reformulate
the algorithm to start with a single basis vector $\bm{x}$.  However,
in all real problems we have solved so far, the dimension~$m$ 
of the basis~$V\in\Cc^{N\times m}$ was at least $O(10)$.}
\label{f:base}
\end{figure}


Unfortunately, the regular MOR solver has certain limitations.
From the solver
description it is clear that, to generate 
the search space $\colspan(V)$ of dimension~$m$, the method needs 
to solve $m$ linear systems~\eqref{mor}.  As numerical experiments
show, the dimension of the search space~$V$ produced by the method
is close to optimal in the following sense.  
If the regular MOR solver provides, for a given accuracy tolerance $\tol$,
a solution space $\colspan(V)$ of dimension~$m$ then dimension~$m$
of the solution space can not be reduced without sacrificing the
reached accuracy level.   
Therefore, if a direct sparse linear solver is applied to solve
systems~\eqref{mor} for selected $k_j$ then,
to produce a solution space of dimension~$m$, a sparse LU factorization
has to be carried out $m$~times.

The situation is even more gloomy if, to solve systems in~\eqref{mor},
we attempt to employ an iterative solver.
Indeed, as can be seen in Figures~\ref{f:eig} and~\ref{f:eig_ilu}, 
where typical spectra of the (preconditioned)
matrices $A_0 + kA_1 + k^2 A_2$ are shown, it is hopeless 
to apply modern preconditioned Krylov subspace solvers for this aim, at least
when conventional ILU-based preconditioners are used.
One could have tried to set up an efficient preconditioner
for $A_0 + kA_1 + k^2 A_2$ (such as an ILU factorization
with a very high fill-in level)
at a specific value of $k\in[k_{\min},k_{\max}]$ and 
to apply it for neighboring values of $k$ in the interval.
Unfortunately, as numerical experiments show,
performance of such a preconditioner is reasonable only for $k$ very
close to the $k$~value, where the preconditioner is computed.
The same holds even if the preconditioner is taken to be a
full (sparse) LU factorization computed at a different $k$.

\begin{figure}
\includegraphics[width=0.48\linewidth]{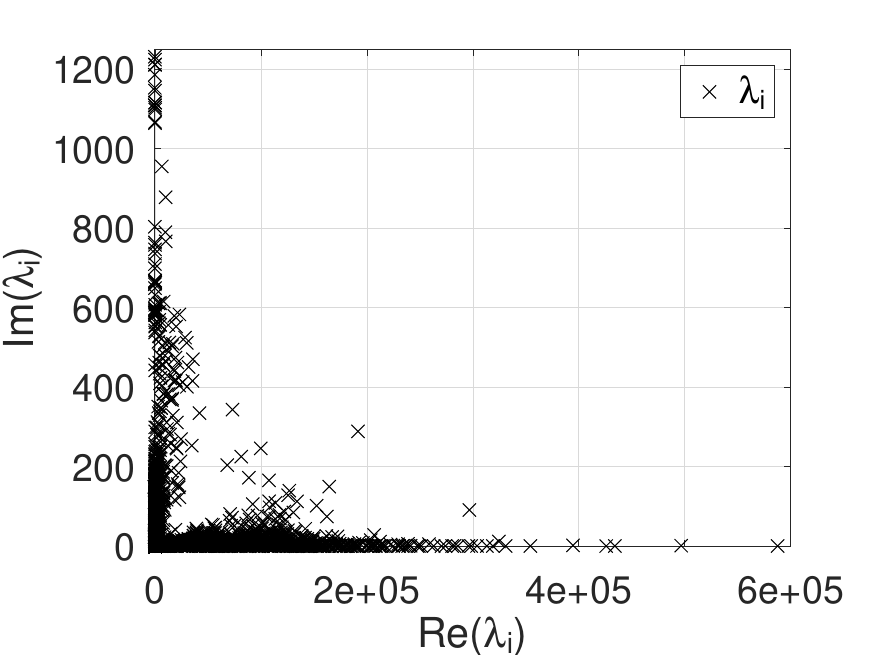}\hfill
\includegraphics[width=0.48\linewidth]{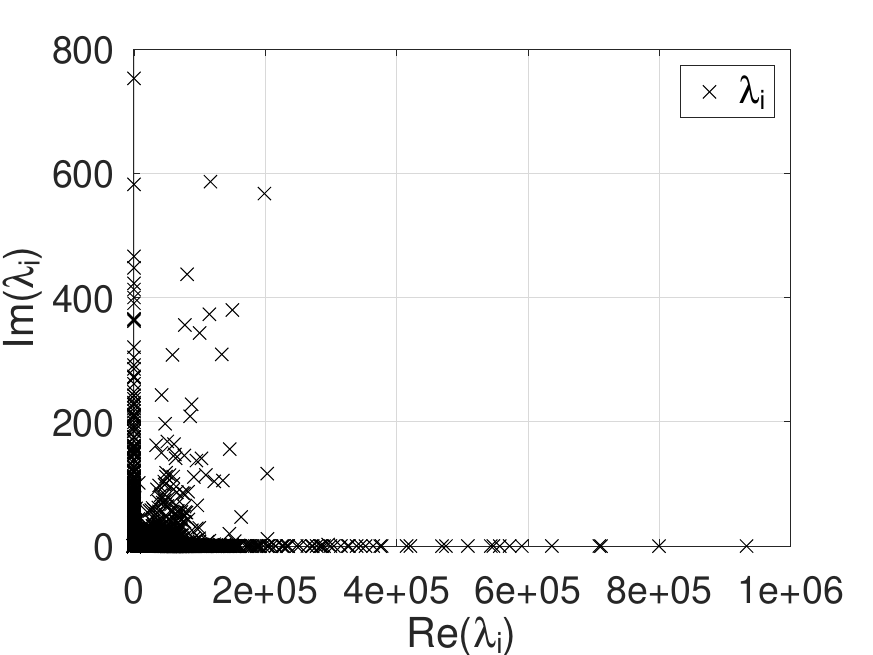}\\
\includegraphics[width=0.48\linewidth]{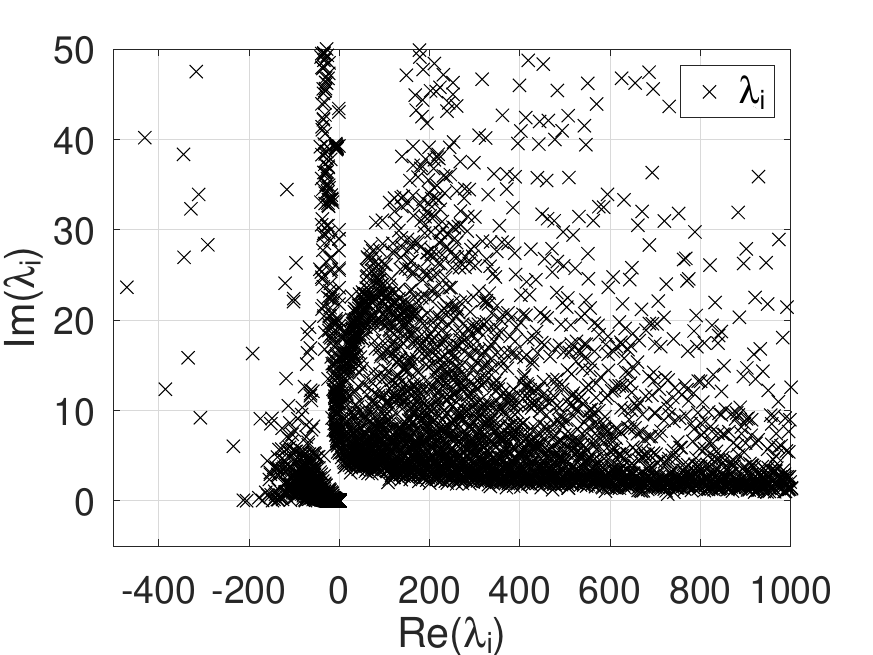}\hfill
\includegraphics[width=0.48\linewidth]{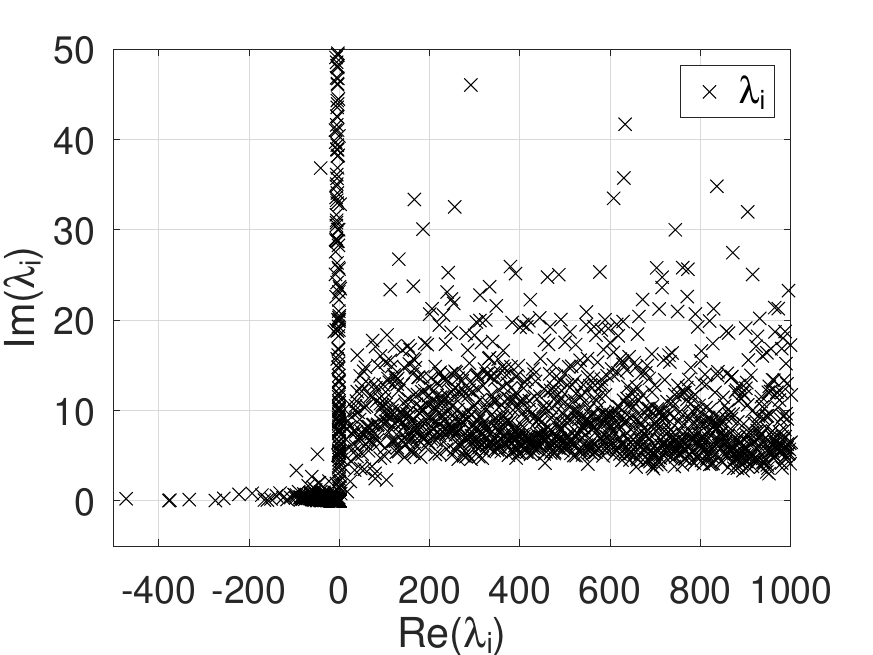}
\caption{The eigenvalues $\lambda_i$ of the matrices 
$A(k_c)=A_0 + k_cA_1 + k_c^2 A_2$
on the complex plane computed at the center point $k_c\equiv(k_{\min}+k_{\max})/2$ 
of the parameter interval~$[k_{\min},k_{\max}]$ for the numerical tests~3 
(top left plot) and~4 (top right plot) described in Section~\ref{s:exp}.
The bottom plots are zooms in the top plots.}
\label{f:eig}  
\end{figure}

Similar limitations are also observed for a MOR solver where 
the minimum residual condition~\eqref{mor_mr} is imposed instead
of the Galerkin condition~\eqref{mor_Gal}.
This indicates that, to design a solver with computational costs 
essentially lower than in the regular MOR solver,
another approach should be followed. 
Such an alternative approach is described in the next section.

\begin{figure}
\includegraphics[width=0.48\linewidth]{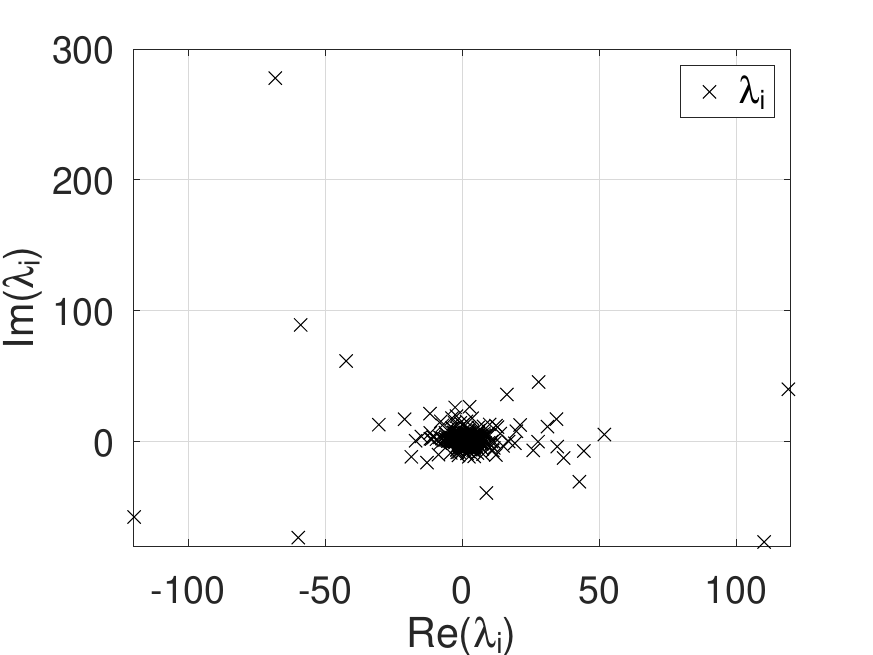}\hfill
\includegraphics[width=0.48\linewidth]{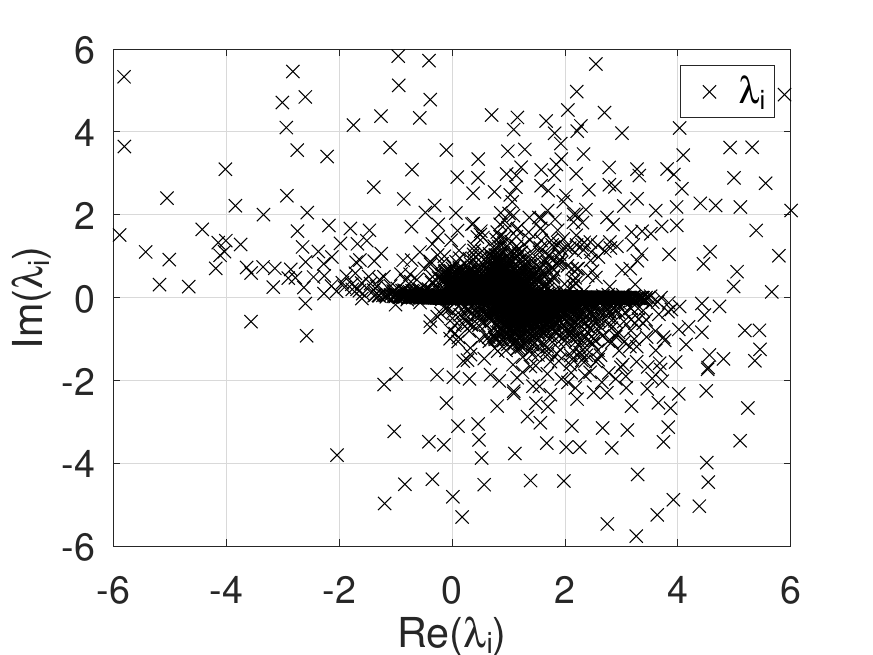}
\caption{Left: the eigenvalues $\lambda_i$ on the complex plane 
of an ILU-preconditioned matrix $A(k_c)=A_0 + k_cA_1 + k_c^2 A_2$,
computed at the center point $k_c\equiv(k_{\min}+k_{\max})/2\in\Rr$ 
for the numerical test~3 (see Section~\ref{s:exp}).
The right plot is a zoom in the left plot. The regular ILU(0)
preconditioner is used.}
\label{f:eig_ilu}  
\end{figure}

\section{ODB: Orthonormal Derivative Basis method}
\label{s:ODB}

\subsection{Derivatives in $k$ and their recursive computation}
\label{s:ODBder}
A natural question which arises and seems to be worthwhile to explore is
whether, 
if a system of family~\eqref{mor} is solved for a certain value~$k$,
we can use derivatives $\bm{x}'(k)$, $\bm{x}''(k)$, \dots{},
along with the solution $\bm{x}(k)$,
to solve~\eqref{mor} for neighboring values~$\tilde{k}$ more efficiently.
The derivatives can be employed, e.g., within a variant of the Hermite 
interpolation.

Indeed, assume for the system family~\eqref{mor} that
\be \label{suppos}
A_0+kA_1+k^2A_2 \quad\mbox{is nonsingular for}\quad 0<k_{\min}\le k\le k_{\max},
\ee
where $k_{\min}$ and $k_{\max}$ are from~\eqref{mor_k}.
Precise conditions which guarantee~\eqref{suppos} are discussed below 
in Section~\ref{s:anal}.
Since nonsingularity condition~\eqref{suppos} implies that
${\bm{x}}(k)$ is analytic on $k_{\min}\le k\le k_{\max}$,
we can differentiate~\eqref{mor} with respect to~$k$.
Repeatedly differentiating~\eqref{mor}, we obtain
\begin{align}
\label{rec1}
& (A_0+kA_1+k^2A_2){\bm{x}}^\prime(k) + (A_1+2kA_2){\bm{x}}(k)={\bm{b}},
\\
\label{recm}
&
\begin{aligned}
(A_0+kA_1+k^2A_2){{\bm{x}}}^{(m)}(k)
 + \binom{m}{1}(A_1+2kA_2) & {{\bm{x}}}^{(m-1)}(k)\\
 + 2\binom{m}{2}A_2        & {{\bm{x}}}^{(m-2)}(k)=0,\qquad
m\ge2,
\end{aligned}
\end{align}
where $\bm{x}^{(m)}(k)$ denotes the $m$th derivative of $\bm{x}(k)$ 
and $\binom{m}{s}$ is the binomial coefficient. 
Relations~\eqref{rec1} and~\eqref{recm} can be rewritten to get 
a recursion for the derivatives
\begin{align}
\label{der1}
& {\bm{x}}^\prime(k) = (A_0+kA_1+k^2A_2)^{-1}
\left[{\bm{b}}-(A_1+2kA_2){\bm{x}}(k)\right],
\\
\label{derm}
& 
\begin{aligned}
{\bm{x}}^{(m)}(k) = -(A_0+kA_1+k^2A_2)^{-1} [ & m(A_1+2kA_2){\bm{x}}^{(m-1)}(k)
\\
 &+ (m-1)mA_2{\bm{x}}^{(m-2)}(k)], \qquad m\ge2.
\end{aligned}
\end{align}
As we see, to compute each derivative, one has to solve a linear system
with the same matrix $A_0+kA_1+k^2A_2$.  
Hence, if a sparse direct linear solver is used to solve the systems 
then just a single sparse
LU factorization suffices to compute all the derivatives.  Likewise,
if an iterative solver is employed then a preconditioner matrix can be
computed just once.  This appears to be crucial for the efficiency of our approach
described below.  

To demonstrate the potential of the Hermite interpolation, 
in~\ref{a:Hermite} we consider a special case of~\eqref{mor}, 
where one of the matrices $A_1$ or $A_2$ is zero.  
As shown there under some additional assumptions, 
in this case Hermite interpolation, approximating $\bm{x}(\tilde{k})$
by linear combinations of 
$\bm{x}(k)$, $\bm{x}'(k)$, $\bm{x}''(k)$, \dots{}, $\bm{x}^{(m)}(k)$,
converges geometrically with growing~$m$ for any allowable $k$ and $\tilde{k}$.  

As numerical experiments presented below show, the positive expectations on the 
Hermite interpolation are not met in practice: we observe an irregular convergence
or no convergence with respect to the derivative order $m$ as well as to
the number of interpolation points.  In the next section we explain
why, in fact, the Hermite interpolation alone can not be a robust solution
approach and propose a way to make it work.
     
\subsection{Orthonormal basis of the derivative subspace}
\label{s:ODBortho}
The unreliable behavior of the algorithm based on the derivatives $\bm{x}^{(m)}(k)$,
which is observed in practice, can be explained 
if we rewrite relations~\eqref{der1} and~\eqref{derm} as
\begin{align} 
\label{x1}
\bm{x}'(k) & = (A_0+kA_1+k^2A_2)^{-1}
\left[{\bm{b}}-(A_1+2kA_2){\bm{x}}(k)\right], 
\\ 
\label{xm}
\frac{ \bm{x}^{(m)}(k)}{m!} & =  -(A_0+kA_1+k^2A_2)^{-1} 
[(A_1+2kA_2) \frac{ \bm{x} ^{(m-1)}(k)}{(m-1)!} 
+ A_2 \frac{ \bm{x} ^{(m-2)}(k)}{(m-2)!} ], \qquad m\ge2, 
\end{align}
and bring the last formula, for $m\ge2$, to a matrix form
\begin{equation}
    \label{x=Bx}
    \begin{bmatrix}
        \bm{x}^{(m)}(k)/m! \\
        \bm{x}^{(m-1)}(k)/(m-1)!
    \end{bmatrix}
    =
    -
    \begin{bmatrix}
        (A_0+kA_1+k^2A_2)^{-1}  & 0     \\
        0                      & I
    \end{bmatrix}
    \begin{bmatrix}
        A_1 + 2 k A_2           & A_2   \\
        -I                      & 0
    \end{bmatrix}
    \begin{bmatrix}
        \bm{x}^{(m-1)}(k)/(m-1)! \\
        \bm{x}^{(m-2)}(k)/(m-2)!
    \end{bmatrix}.
\end{equation}
In other words, 
computing the derivatives $\bm{x}^{(m)}(k)$ can be seen, for $m\ge 2$,
as a power iteration for the matrix 
\begin{equation}
\label{B}  
B = -
    \begin{bmatrix}
        (A_0+kA_1+k^2A_2)^{-1}  & 0     \\
        0                       & I
    \end{bmatrix}
    \begin{bmatrix}
        A_1 + 2 k A_2           & A_2   \\
        -I                      & 0
    \end{bmatrix}.
\end{equation}
This process is well known to converge 
(up to a normalization) to an eigenvector of  
\(B\) corresponding to the largest in magnitude eigenvalue and, thus, 
the vectors $\bm{x}^{(m)}(k)$
tend to become collinear with growing \(m\).
Hence, as we see,
these vectors can not be used as a numerically reliable MOR solution basis.

Fortunately, the well known problem with linear dependence in the power iteration
has a well known remedy: the Arnoldi process,
see, e.g.,~\cite{Saad2003book,vdVorst2003book},
which produces an orthonormal basis of the Krylov subspace generated
by the matrix~$B$.
Note that applying the Arnoldi process is much more numerically stable than computing the vectors $\bm{x}^{(1)}(k)$, \dots, $\bm{x}^{(m)}(k)$
and orthogonalizing them afterwards.

We now describe our Arnoldi procedure to compute an orthonormal basis 
for the subspace spanned by 
$\bm{x}(k)$, $\bm{x}'(k)$, \dots{}, $\bm{x}^{(m)}(k)$ in more detail.
With $\bm{x}^{(-1)}(k)$ formally defined as a zero vector,
we introduce
\begin{equation}
\label{zm}  
\bm{z}_m=
\begin{bmatrix}
\bm{x}^{(m)}(k) \\ \bm{x}^{(m-1)}(k)
\end{bmatrix}, \quad m\ge 0,
\quad\text{with}\quad \bm{x}^{(-1)}(k):=0.
\end{equation}
Rewriting the first recursion step~\eqref{x1} as
\begin{align*}
      \bm{x}'(k) & = (A_0+kA_1+k^2A_2)^{-1}[{\bm{b}}-(A_1+2kA_2){\bm{x}}(k)] \\ 
	         & = \bm{x}(k) - (A_0+kA_1+k^2A_2)^{-1} (A_1+2kA_2){\bm{x}}(k),
\end{align*}
and using~\eqref{B},\eqref{zm}, 
we see that
\begin{equation}
\label{z1}
    \bm{z}_1 = 
    \begin{bmatrix}
        \bm{x}' (k) \\
        \bm{x} (k)
    \end{bmatrix}
    =
    \begin{bmatrix}
        \bm{x} (k) \\
        0
    \end{bmatrix}
    -
    \begin{bmatrix}
        (A_0+kA_1+k^2A_2)^{-1}  & 0     \\
        0                       & I
    \end{bmatrix}
    \begin{bmatrix}
        A_1 + 2 k A_2           & A_2   \\
        -I                      & 0
    \end{bmatrix}
    \begin{bmatrix}
        \bm{x} (k)\\
        0
    \end{bmatrix}
    =
    \bm{z}_0
    +
    B \bm{z}_0.
\end{equation}
Hence, our recursion~\eqref{x1},\eqref{xm} to compute
the derivatives~$\bm{x}^{(m)}(k)$ reads, for the vectors $\bm{z}_m$, as
\begin{equation}
\label{z=Bz}
\bm{z}_1 = \bm{z}_0 + B \bm{z}_0,
\quad
\bm{z}_m = B \bm{z}_{m-1},\quad m=2,3,\dots.  
\end{equation}
It is not difficult to see that replacing \eqref{z=Bz}
by a modified recursion
\begin{equation}
\label{z=Bz1}
\bm{z}_m = B \bm{z}_{m-1},\quad m=1,2,\dots,
\end{equation}
leaves
$\mathrm{span}( \bm{z}_0, \bm{z}_1, \dots , \bm{z}_m )$ unchanged.
In practice, the recursion~\eqref{z=Bz1} is slightly easier to work with,
as there is no need to handle the first recursion step individually.

As numerical experiments show, it turns out to be advantageous, 
in terms of the accuracy reached, to apply the Arnoldi process to a power iteration 
with a modified, balanced and shifted, matrix~$B$.
To balance~$B$, we rewrite it as
\begin{equation*}
\label{B_fact}  
\begin{aligned}
    B &= 
    -
    \begin{bmatrix}
        (A_0+kA_1+k^2A_2)^{-1}  & 0     \\
        0                       & I
    \end{bmatrix}
    \begin{bmatrix}
        A_1 + 2 k A_2           & A_2   \\
        -I                      & 0
    \end{bmatrix}
\\
    &=
    -
    \frac{1}{k}
    \begin{bmatrix}
        (A_0+kA_1+k^2A_2)^{-1}  & 0     \\
        0                       & I
    \end{bmatrix}
    \begin{bmatrix}
        I  & 0     \\
        0  & kI\\
    \end{bmatrix}
    \begin{bmatrix}
        kA_1 + 2 k^2 A_2       & k^2 A_2   \\
        -I                     & 0
    \end{bmatrix}
    \begin{bmatrix}
        I  & 0     \\
        0  & \frac{1}{k} I\\
    \end{bmatrix}
\\
    &=
    -
    \frac{1}{k}
    \begin{bmatrix}
        I  & 0     \\
        0  & kI\\
    \end{bmatrix}
    \begin{bmatrix}
        (A_0+kA_1+k^2A_2)^{-1}  & 0     \\
        0                       & I
    \end{bmatrix}
    \begin{bmatrix}
        kA_1 + 2 k^2 A_2       & k^2 A_2   \\
        -I                      & 0
    \end{bmatrix}
    \begin{bmatrix}
        I  & 0     \\
        0  & \frac{1}{k} I\\
    \end{bmatrix} 
= \frac{1}{k} P C P^{-1}.
\end{aligned}
\end{equation*}
Here
\be \label{defC}
C = -
    \begin{bmatrix}
        (A_0+kA_1+k^2A_2)^{-1}  & 0     \\
        0                       & I
    \end{bmatrix}
    \begin{bmatrix}
        kA_1 + 2 k^2 A_2       & k^2 A_2   \\
        -I                      & 0
    \end{bmatrix},
\qquad
  P = 
    \begin{bmatrix}
        I  & 0     \\
        0  & kI\\
    \end{bmatrix},
\ee
and $C$ can be seen as a balanced version of~$B$.
Note that the matrices $A_1$ and $A_2$ appear in~$C$ 
accompanied by the factors $k$ and $k^2$, respectively,
in the same way as in the original problem~\eqref{mor}.
Thus, working with~$C$, we handle terms of the same order 
of magnitude as in~\eqref{mor}.

The next important observation is that replacing the matrix~$B$ 
in the recursion~\eqref{z=Bz1} by the matrix~$C$,
while simply neglecting $1/k$, $P$, and $P^{-1}$, leads, in exact arithmetic,
to exactly the same MOR solution basis.
Indeed, multiplication with the matrices $P$ and $P^{-1}$
affects only the second half of the recursion vector entries,
whereas the second part in the initial recursion vector $\bm{z}_0$ is zero,
see~\eqref{zm}.
Thus, the MOR solution basis for the subspace 
$\mathrm{span} (\bm{x}(k), \bm{x}'(k), \dots, \bm{x}^{(m)}(k) )$,
formed by the first $N$ entries of $\bm{z}_j$, is left unaffected.

\begin{figure}
\hrule
Given basis dimension $m$, tolerance \tol{} and point~$k=\hat{k}$, 
for $C+I$ defined in~\eqref{CplusI}, carry out the steps:
\begin{enumerate}
    \item Set \(\bm{q}_1 = \begin{bmatrix} \bm{x}(k) \\ 0 \end{bmatrix}\), \(Q_1 = \begin{bmatrix} \bm{q}_1 \end{bmatrix}\), $\jstart=2$, $\jend=m$.
    \item For $j=\jstart,\dots,\jend$
        \begin{enumerate}
            \item Compute \(\bm{p}_j = (C + I) \bm{q}_{j-1}\).
            \item Orthogonalize \(\bm{p}_j\) with respect to \(Q_{j-1}\)
(using, e.g., the Gram--Schmidt process\\ 
\(\hat{\bm{q}}_j = \left(I - Q_{j-1} Q_{j-1}^*\right)\bm{p}_j\)), 
denote the obtained vector as \(\hat{\bm{q}}_j\).
            \item Set $\bm{q}_j = \frac{\hat{\bm{q}}_j}{\|\hat{\bm{q}}_j\|}$, 
                      $Q_j = \begin{bmatrix} Q_{j-1} & \bm{q}_j \end{bmatrix}$.
        \end{enumerate}
    \item Take the first $N$ rows of the matrix \(Q_m\) as the MOR basis matrix $V$,
mark all $k_j$, $j=1,\dots,J$, as ``unsolved'' and carry out the steps~2, 3 of Algorithm in Figure~\ref{f:base} with tolerance~$\tol$.
    \item If all $k_j$ are marked as ``solved'' then STOP;\\ 
         otherwise increase $m:=m+1$, set $\jstart=m$, $\jend=m$ and GO TO step~2 .
\end{enumerate}
\hrule
\caption{The orthonormal derivative basis (ODB) method to solve the system 
family~\eqref{mor} for a given relative accuracy tolerance~$\tol$}
\label{f:odb}
\end{figure}

Finally, in addition to balancing~$B$ in~\eqref{z=Bz1}
by replacing it by the matrix~$C$, in practice, we apply the Arnoldi
process to a recursion with a shifted matrix $C+I$.
As evidenced by numerical tests, this allows in some cases
to attain a higher accuracy than when working with~$C$.
Clearly, due to the shift invariance of Krylov subspaces,
the bases produced for $C + I$ and $C$ are the same in exact
arithmetic.  We have
\begin{multline}
\label{CplusI}
C + I = 
-
    \begin{bmatrix}
        (A_0+kA_1+k^2A_2)^{-1}  & 0     \\
        0                       & I
    \end{bmatrix}
    \begin{bmatrix}
        kA_1 + 2 k^2 A_2       & k^2 A_2   \\
        -I                      & 0
    \end{bmatrix}
    +
    \begin{bmatrix}
        I  & 0     \\
        0                       & I
    \end{bmatrix}\\
    =
    \begin{bmatrix}
        I - (A_0+kA_1+k^2A_2)^{-1}(A_0 + kA_1 + k^2 A_2 - A_0 + k^2 A_2)  & -k^2 A_2     \\
        I                       & I
    \end{bmatrix}\\
    =
    \begin{bmatrix}
        (A_0+kA_1+k^2A_2)^{-1}  & 0     \\
        0                       & I
    \end{bmatrix}
    \begin{bmatrix}
        A_0 - k^2 A_2       & -k^2 A_2   \\
        I                      & 0
    \end{bmatrix},
\end{multline}
where it can checked that the blocks $A_0 - k^2 A_2$  and $-k^2 A_2$ in the second factor of $C+I$ have
a positive definite Hermitian part (recall that $-A_2$ is a mass matrix).

Thus, the observed better accuracy attained when working with $C+I$ 
can probably be attributed to the definiteness of the blocks.
Indeed, this property can, in general, be useful in avoiding
rounding errors caused, for instance, by cancellation 
effects, see, e.g.,~\cite{Higham:Accuracy}.     
The Arnoldi based algorithm for the power iteration with $C+I$, 
which we call the orthonormal derivative basis (ODB) method, is presented
in Figure~\ref{f:odb}.

\begin{remark}
\label{r:svd}
In some cases the MOR basis $V$ delivered by the ODB algorithm has a larger 
dimension than the basis built by the regular MOR algorithm. 
Using conventional techniques,
such as truncated singular value decomposition (SVD), the dimension
of the MOR basis produced by the ODB method can easily be reduced afterwards,
with low computational costs.
A possible algorithm for such a reduction, which is based on truncated SVD
and replaces the computed basis matrix 
$V\in\Cc^{N\times m}$ by $V_{\mathrm{red}}\in\Cc^{N\times\hat{m}}$
with $\hat{m}<m$, is described in~\ref{a:SVD}.   
\end{remark}
 


\subsection{A faster convergence with a complex $k$}
\label{s:anal}
The regular MOR algorithm (cf.\ Figure~\ref{f:base}) builds an orthonormal MOR solution 
basis consisting of the solutions $\bm{x}(k)$
computed at different $k$ within the given range $[k_1,k_J]$. 
In contrast, the ODB algorithm constructs a solution basis
using a solution $\bm{x}(\hat{k})$ and its derivatives $\bm{x}'(\hat{k})$, 
$\bm{x}''(\hat{k})$, \dots{} computed at a single
point $\hat{k}$.
How large the ODB basis should be to 
span approximate solutions 
for all $k\in [k_1,k_J]$ (in other words, how fast the ODB method converges) depends, in general, 
on the point $\hat{k}$ where the ODB basis is computed.  
Observe that this point $\hat{k}$ does not have to belong to the interval $[k_1,k_J]$
and can, in principle, be taken anywhere in the complex plane,
as long as the matrix $A_0+\hat{k}A_1+\hat{k}^2A_2$ is nonsingular.
For which $k$ the matrix $A_0+kA_1+k^2A_2$ is nonsingular is established
by the proposition below.

Consider the numerical range of the
matrix $A({k})=A_0+{k} A_1+{k}^2A_2$, i.e., the set
\[ 
\{ z = (A({k})\bm{x},\bm{x})\in\Cc \; | \; \bm{x}\in\Cc^N,\; 
(\bm{x},\bm{x})=1 \}.
\]
It is well known and easy to check that the numerical range of a matrix contains
all its eigenvalues.  The following proposition holds.
   
\begin{proposition}
\label{prop1}
The numerical range of the matrix $A(k)=A_0+kA_1 + k^2A_2$ does not contain zero 
(and, hence, the matrix $A(k)$ is nonsingular) provided that
\begin{equation}
\label{good_k}
k\in \left\{ \, 
z\in\Cc \; \bigl| \; \mathrm{Im}(z)\leqslant \min\{0, \mathrm{Re}(z)\} 
\;\mathrm{and}\; z\not\in\Rr 
\, \right\}.
\end{equation}
\end{proposition}

\begin{proof}
Taking into account~\eqref{mor_prop}, 
for all $k=\alpha + i\beta$, $\alpha,\beta\in\Rr$, we have
\begin{equation}
\label{Ak_re_im}
\begin{aligned}
A(k) &=A_0 + k A_1 + k^2 A_2 = P + i Q, 
\\
P &= P^T = A_0 -\beta S + (\beta^2-\alpha^2)M - 2\alpha\beta M_\delta  \in\Rr^{N\times N},
\\
Q &= Q^T = \alpha S - 2\alpha\beta M + (\alpha^2-\beta^2)M_\delta \in\Rr^{N\times N}.
\end{aligned}
\end{equation}
Since the matrices $P$ and $Q$ are real symmetric, their numerical ranges 
are, respectively, real intervals
\[ 
[\lambda_{\min}(P),\lambda_{\max}(P)]\subset\Rr,
\qquad
[\lambda_{\min}(Q),\lambda_{\max}(Q)]\subset\Rr,
\]
where 
$\lambda_{\min}(X)$ and $\lambda_{\max}(X)$
correspondingly denote the smallest and largest eigenvalues of a matrix $X$.
Therefore, it is easy to check that 
the numerical range of $A(k)$ is contained in a rectangle in the complex plane
$$
[\lambda_{\min}(P),\lambda_{\max}(P)] \times 
i [\lambda_{\min}(Q),\lambda_{\max}(Q)] \subset \Cc.
$$
Hence, to guarantee that the numerical range of $A(k)$
does not contain the origin, it is sufficient to require that
either $P$ or $Q$ are positive or negative definite, i.e.,
at least one of the following four conditions holds:
$$
\lambda_{\min}(P)>0,\quad \lambda_{\max}(P)<0,
\quad
\lambda_{\min}(Q)>0,\quad \lambda_{\max}(Q)<0.
$$
We have $\lambda_{\min}(P)>0$ (i.e., $P$ is positive definite), 
if the following conditions on $\alpha$ and $\beta$ simultaneously hold:
$$
\beta\leqslant 0, 
\quad
-2\alpha\beta\geqslant 0,
\quad
\beta^2 - \alpha^2 > 0, 
$$ 
which are equivalent to
\begin{equation}
\label{Q4a}
0 \leqslant \alpha < -\beta.
\end{equation}
Note that it is, in general, difficult to achieve $\lambda_{\max}(P)<0$  
because $P$ contains a real symmetric positive semidefinite matrix $A_0$
(the stiffness matrix).
Furthermore, it is easy to check that the condition $\lambda_{\min}(Q)>0$ is satisfied
provided
\begin{equation}
\label{Q4b}
0 < -\beta \leqslant \alpha, 
\end{equation}
and the condition $\lambda_{\max}(Q)<0$ holds as soon as
\begin{equation}
\label{Q3b}
\beta \leqslant \alpha < 0.
\end{equation}
Uniting the regions~\eqref{Q4a}, \eqref{Q4b}, and
\eqref{Q3b}, we obtain the sought after condition~\eqref{good_k}.

\end{proof}

\begin{figure}
\centerline{\includegraphics[width=0.5\linewidth]{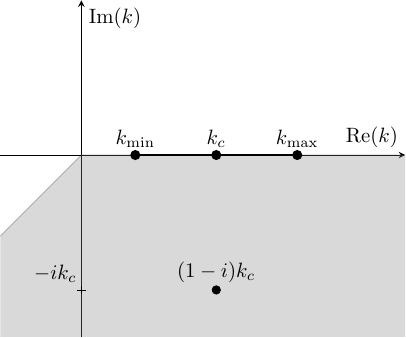}}
\caption{Taking $k\in\Cc$ in the gray region (excluding real nonnegative values $k\geqslant 0$)
leads to a nonsingular matrix $A_0+kA_1 + k^2A_2$,
see relation~\eqref{good_k}.
The problem has to be solved for $k\in[k_{\min},k_{\max}]$.
We suggest to apply the ODB method at $k=\hat{k}\equiv(1-i)k_c$, with $k_c$ being 
the interval center $k_c=(k_{\min}+k_{\max})/2$.}
\label{f:3over8}
\end{figure}

The connected set~\eqref{good_k} where the matrix is guaranteed to be nonsingular 
is shown in Figure~\ref{f:3over8}.
As we see, the interval $[k_1,k_J]$ lies on the boundary of 
this region.
Therefore, taking the point $\hat{k}$, where the ODB method is applied, 
away from the region boundary should
lead to a wider convergence domain in our ODB method. 
This is confirmed in practice.
Indeed, it turns out that in all the numerical tests we carried out 
a good value for $\hat{k}$ is 
\begin{equation}
\label{khat}  
\hat{k}\approx k_c - ik_c,\qquad
k_c=(k_{\min}+k_{\max})/2
\end{equation}
where $k_c$ is the center point of the interval $[k_{\min},k_{\max}]=[k_1,k_J]$ 
and $i$ is the imaginary unit ($i^2=-1$).
With this $\hat{k}$ the ODB method converges in most tests 
considerably better (and never much worse) than for $\hat{k}\in[k_1,k_J]$. 

\begin{remark}
Choosing $\hat{k}\approx k_c - ik_c$ is also advantageous if,
within the ODB method, 
a (properly preconditioned) iterative solver is employed to solve the linear system with 
the matrix $A(\hat{k}) = A_0 + \hat{k} A_1 + \hat{k}^2 A_2$.
Then we typically observe that its convergence is much better for
$\hat{k}\approx k_c - ik_c$ than for $\hat{k}\in[k_1,k_J]$. 
This can be explained by the fact that for 
a $\hat{k}=\alpha + i \beta$ satisfying~\eqref{Q4a}
the matrix $\mathrm{Re}(A(\hat{k}))=P$ (cf.~\eqref{Ak_re_im}) 
is positive definite, so that 
the numerical range of $A(\hat{k})$ lies in the complex right half-plane.
\end{remark}

To illustrate our observations, in Figure~\ref{f:eigC} we plot 
the eigenvalues of the matrices~$A(\hat{k})$ from the numerical
tests~3 and~4 considered below in Section~\ref{s:exp}.
It is instructive to compare these plots with those given in
Figure~\ref{f:eig}.
In addition, the spectral boundaries 
$\min(\mathrm{Re}(\lambda))$, $\max(\mathrm{Re}(\lambda))$, 
$\min(\mathrm{Im}(\lambda))$, $\max(\mathrm{Im}(\lambda))$
for the eigenvalues~$\lambda$ computed at the interval center $k_c$ 
and at the point $\hat{k}=(1-i)k_c\in\Cc$ are listed in Table~\ref{t:eig}.

\begin{figure}
\includegraphics[width=0.48\linewidth]{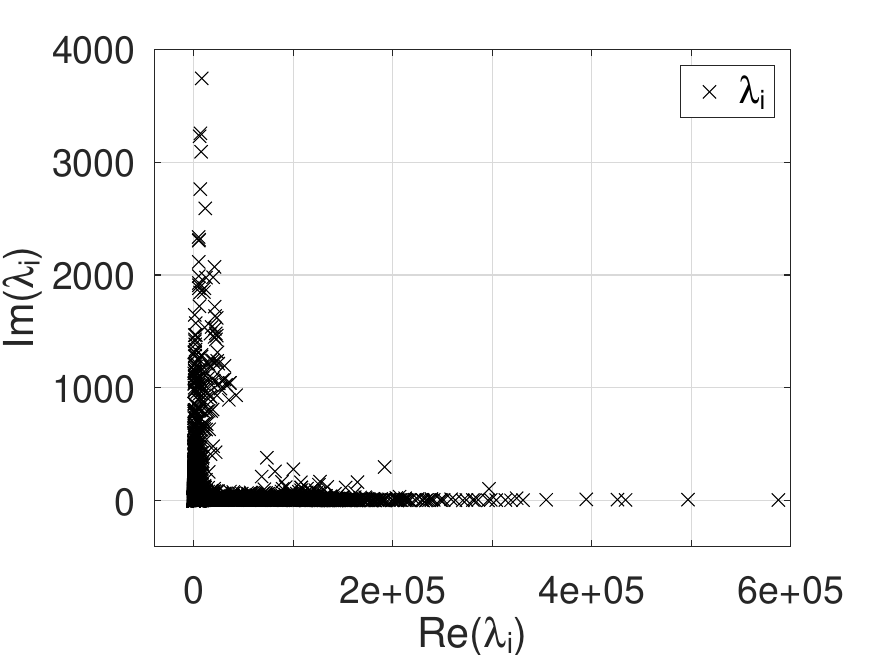}\hfill
\includegraphics[width=0.48\linewidth]{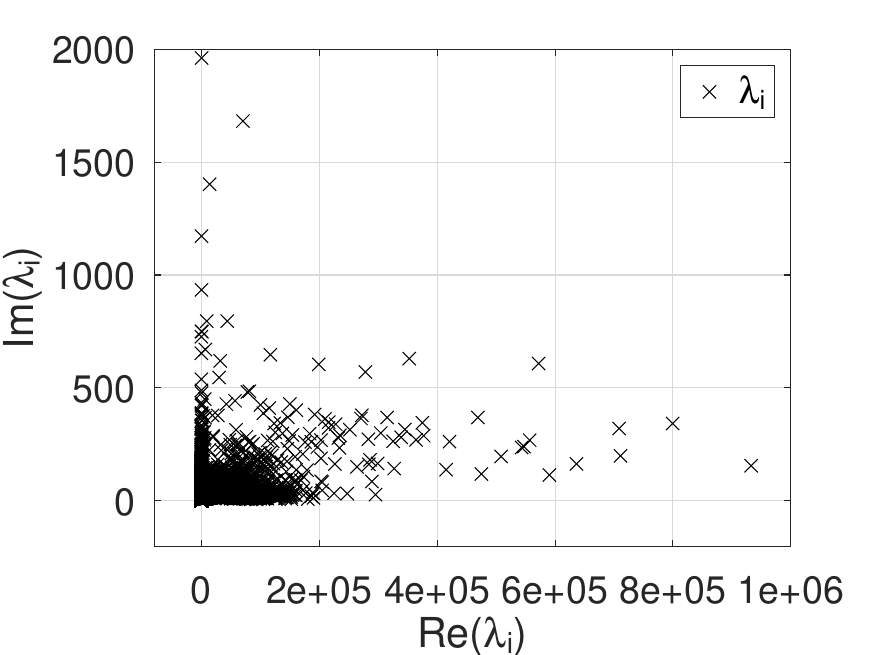}\\
\includegraphics[width=0.48\linewidth]{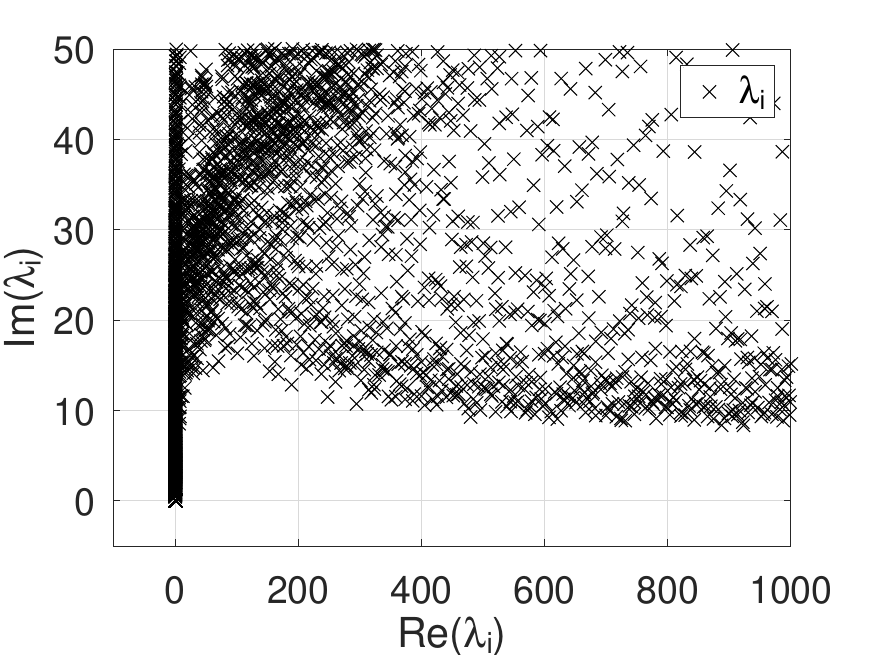}\hfill
\includegraphics[width=0.48\linewidth]{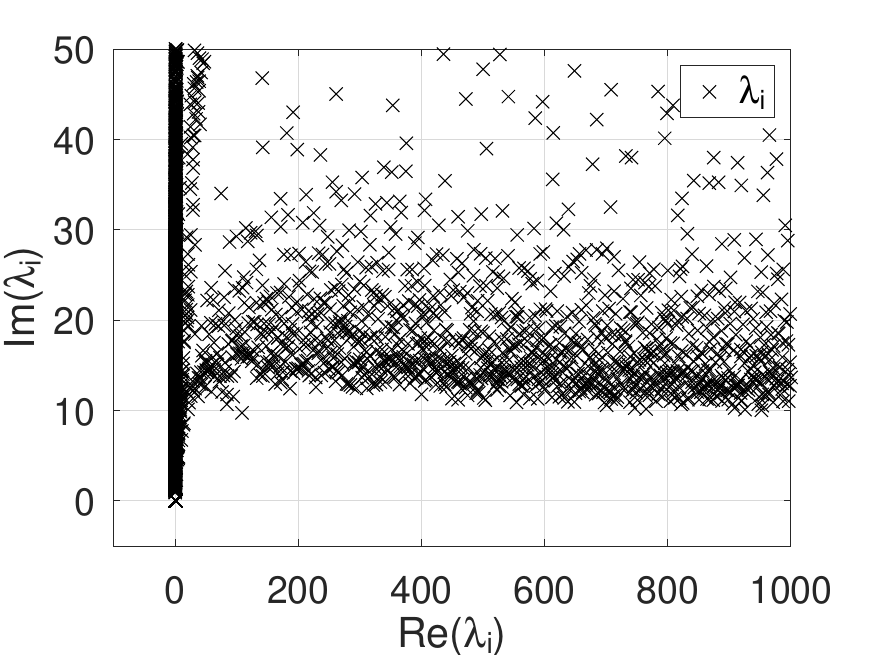}
\caption{The eigenvalues $\lambda_i$ of the matrices 
$A(k_c)=A_0 + k_cA_1 + k_c^2 A_2$
on the complex plane computed at the point $\hat{k}=(1-i)k_c\in\Cc$,
with $k_c\equiv(k_{\min}+k_{\max})/2$, 
for the numerical tests~3 
(top left plot) and~4 (top right plot) described in Section~\ref{s:exp}.
The bottom plots are zooms in the top plots.}
\label{f:eigC}  
\end{figure}

\begin{table}
\caption{Spectral boundaries 
$\min(\mathrm{Re}(\lambda))$, $\max(\mathrm{Re}(\lambda))$, 
$\min(\mathrm{Im}(\lambda))$, $\max(\mathrm{Im}(\lambda))$
for the eigenvalues~$\lambda$ computed at the interval center $k_c$ 
and at the point $\hat{k}=(1-i)k_c\in\Cc$, for test problems~3 and~4
introduced in Section~\ref{s:exp}}
\begin{center}
\begin{tabular}{|c|c|c|}
\hline\hline
test & $k$    & 
$\min(\mathrm{Re}(\lambda))$, $\max(\mathrm{Re}(\lambda))$, 
$\min(\mathrm{Im}(\lambda))$, $\max(\mathrm{Im}(\lambda))$ 
\\\hline
3 & 
$k_c\in\Rr$             &{\tt-9.714e+02}, {\tt5.875e+05}, {\tt0.000e+00}, {\tt1.231e+03}\\
& 
$\hat{k}=(1-i)k_c\in\Cc$& {\tt\ 7.008e-05}, {\tt5.875e+05}, {\tt0.000e+00}, {\tt3.741e+03}
\\\hline
4 & 
$k_c\in\Rr$             &{\tt-9.723e+02}, {\tt9.326e+05}, {\tt0.000e+00}, {\tt7.530e+02}\\
& 
$\hat{k}=(1-i)k_c\in\Cc$& {\tt\ 3.066e-04}, {\tt9.327e+05}, {\tt0.000e+00}, {\tt1.961e+03}
\\\hline  
\end{tabular}
\end{center}
\label{t:eig}  
\end{table}

\section{Numerical experiments}
\label{s:exp}
In this section we compare our ODB method with the regular MOR solver,
described in Section~\ref{s:POD} and outlined in Figure~\ref{f:base}.
This is done on a set of four test problems outlined in Table~\ref{t:tests}.
We run the tests on a Linux PC with a ten-core 3.30GHz~CPU
and 64\,Gb memory,  
in GNU Octave version 9.2.0 (\url{https://octave.org/}). 
Due to the relatively large size of the problems (up to $N=6\,594\,156$),
the Octave codes are implemented with calls to an external direct sparse
LU linear solver.
In all test runs we use a realistic relative accuracy value $\tol=10^{-2}$.

\begin{table}
\caption{An overview of the test problems}
\label{t:tests}
\begin{center}
\begin{tabular}{|c|c|c|c|}
\hline\hline
test   & problem    &  the range $[k_{\min},k_{\max}]$ and \# points~$J$ for the $k_j$ values,  &  application  \\  
        & size, $N$  &  $k_j=k_{\min} + \frac{j-1}{J-1}(k_{\max}-k_{\min})$, $j=1,\dots,J$       
                                      &               \\
\hline
1 & $601\,160$ & $[k_{\min},k_{\max}]=[2.094,418.879] \; (\; = [0.1, 20]\, \text{GHz}\;)$, $J=101$ & antenna design 
\\\hline
2 & $6\,594\,156$ & $[k_{\min},k_{\max}]=[52.988,56.758] \; (\; = [2.52, 2.71]\, \text{GHz}\;)$, $J=201$ & filter design 
\\\hline
3 & $16\,330$ & $[k_{\min},k_{\max}]=[2.094,418.879] \; (\; = [0.1, 20]\, \text{GHz}\;)$, $J=101$ & antenna design 
\\\hline
4 & $19\,674$ & $[k_{\min},k_{\max}]=[2.094,418.879] \; (\; = [0.1, 20]\, \text{GHz}\;)$, $J=101$ & antenna design 
\\\hline
\end{tabular}
\end{center}
\end{table}

The results of the test runs are presented in Table~\ref{t:all}.
In the table, the number of LU~factorizations, iterative steps
and the resulting MOR basis dimensions are shown for both the
ODB and the regular MOR methods.  The residual norm displayed in the
table for both methods is computed as the maximal relative residual norm
of the linear system family~\eqref{mor} for all $k_j$, i.e., as
\begin{equation}
\label{relres}  
\max_{j=1,\dots,J} 
\frac{\|k_j\bm{b} - (A_0+k_jA_1 + k_j^2A_2)Vy_j\|}{\|k_j\bm{b}\|},
\end{equation}
where $Vy_j$ is the computed MOR approximation to the solution~$\bm{x}_j$ 
of system~\eqref{mor} for $k=k_j$, see~\eqref{x=Vy}.
Furthermore, in Table~\ref{t:all} we indicate whether the MOR basis produced
by the ODB method is reduced by truncated SVD, see Remark~\ref{r:svd}.
The truncated SVD basis reduction is described in~\ref{a:SVD} and 
is in all cases run with a tolerance $2\cdot 10^{-5}$
(meaning that the truncated singular values are smaller than the tolerance
times the largest singular value).
We emphasize that the costs for applying the truncated SVD are negligible
and, in practice, such a basis dimension reduction can be applied by
default.
In the last column of Table~\ref{t:all} the CPU gain factor delivered
by the ODB method with respect to the regular MOR method is shown.
To give an impression on how this gain is obtained, we report 
a break-up of the CPU times spent by each solver in test~2.
The regular MOR algorithm carries out a sparse LU~factorization
of the matrix $A_0+k_jA_1 + k_j^2A_2$ for 13~different values of $k_j$, 
with each LU factorization requiring approximately 108~s of CPU time.
Hence, the total CPU time needed by the regular ODB method is a at least
$13\cdot 108\approx 1400$~s.
In contrast, the CPU time in the ODB method is spent mainly for a single 
sparse LU~factorization (computed at $\hat{k}$, cf.~\eqref{khat})
and $m=20$~applications of the computed LU, each costing approximately
2~s of CPU time.  Therefore, the total CPU time of the ODB method is approximately
$108+20\cdot 2=148$~s.  
This gives the CPU time gain factor $1400/148\approx 9.46$ reported in the table.

\begin{table}
\caption{Required number of LU~factorizations, 
MOR subspace dimensions and obtained accuracies in the ODB
and regular MOR solver}
\label{t:all} 
\begin{center}
\begin{tabular}{|c|c|c|c|c|c|c|}
\hline\hline
     & \# LU's     & \# steps     & SVD basis & basis dim.  & res.norm    & CPU time\\
test & ODB /       & ODB /        & reduction & ODB /       & ODB /       & gain  \\
     & regular MOR & regular MOR  & in ODB    & regular MOR & regular MOR & factor\\
\hline
1 & 1 / 11  & 25 / 11 & yes & 20 / 11 & {\tt3.96e-3} / {\tt9.94e-3} & 4.32 \\
\hline
2 & 1 / 13  & 12 / 13 & no  & 12 / 13 & {\tt6.11e-3} / {\tt9.96e-3} & 9.46 \\
\hline
3 & 1 /  6  & 11 / 6  & no  & 11 /  6 & {\tt6.14e-3} / {\tt9.99e-3} & 2.81 \\
\hline
4 & 1 / 10  & 25 / 10 & yes & 12 / 10 & {\tt4.53e-3} / {\tt9.59e-3} & 2.71 \\
\hline
\end{tabular}
\end{center}
\end{table}

In Figures~\ref{f:mod1res}--\ref{f:mod6res} we plot relative residual norms 
$\|k_j\bm{b} - (A_0+k_jA_1 + k_j^2A_2)Vy_j\|/\|k_j\bm{b}\|$
(see~\eqref{relres}) versus the 
frequency range $[\omega_{\min},\omega_{\max}]$ with the frequency values
$\omega$ defined as $\omega = 3k/(20\pi)$.
The plots nicely illustrate the key difference of the two solvers:
while the solution quality delivered by the regular MOR solver essentially 
depends on the number of points where the system family~\eqref{mor} is solved,
the ODB solver is free from this restriction and delivers a solution
basis suitable for the whole frequency range.
Figure~\ref{f:flt5res1} shows convergence of the ODB method for test~2, there
the relative residual norms versus the frequency range are plotted 
for increasing iteration number~$m$.

\begin{figure}
\centerline{\includegraphics[width=0.6\linewidth]{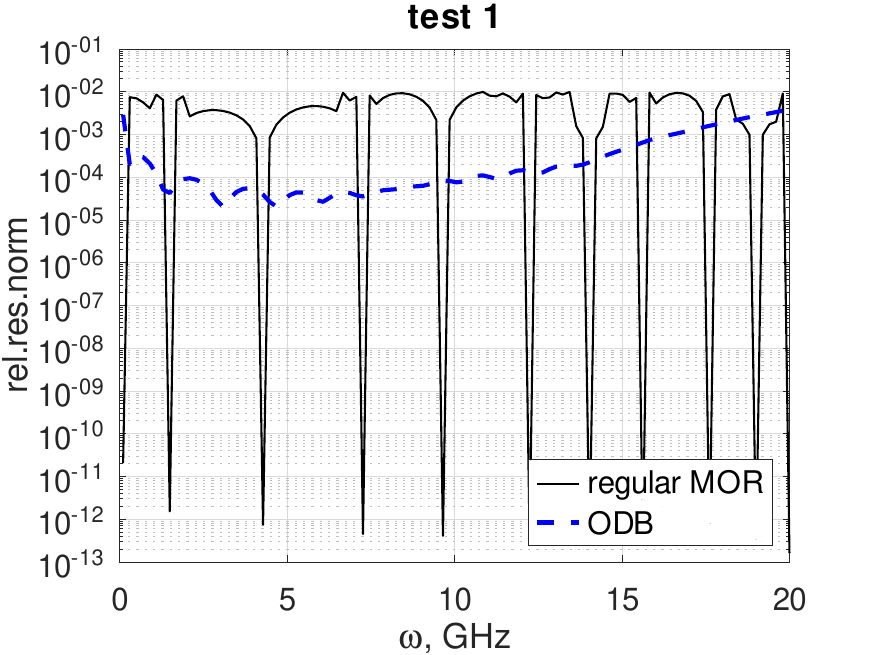}}
\caption{Relative residual norms versus frequency in the regular MOR and 
ODB solvers for test~1.
The ODB solver reaches the required tolerance $10^{-2}$ after 
$m=25$ steps. The ODB basis is then reduced to dimension~20 by SVD
truncated with tolerance $2\cdot 10^{-5}$.}
\label{f:mod1res}  
\end{figure}

\begin{figure}
\centerline{\includegraphics[width=0.6\linewidth]{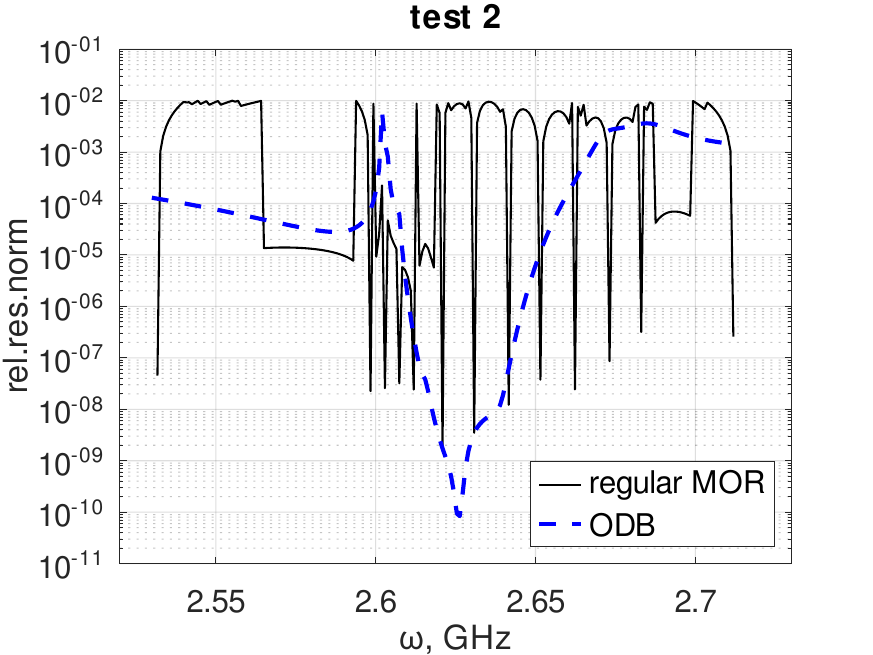}}
\caption{Relative residual norms versus frequency in the regular MOR and 
ODB solvers for test~2.  
The ODB solver reaches the tolerance $10^{-2}$ after 
$m=12$ steps, with no dimension reduction by truncated SVD.}
\label{f:flt5res}
\end{figure}

\begin{figure}
\centerline{\includegraphics[width=0.6\linewidth]{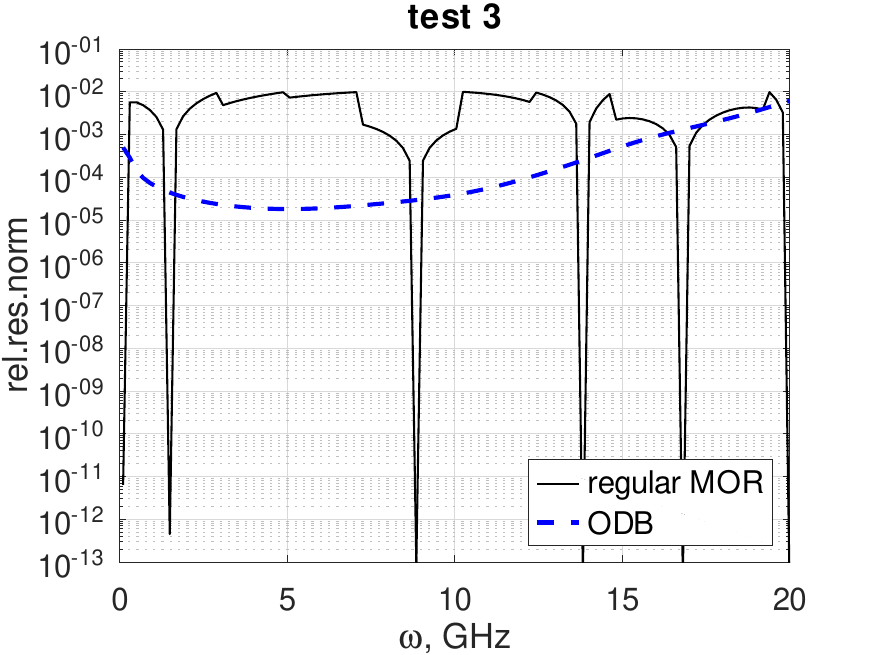}}
\caption{Relative residual norms versus frequency in the regular MOR and 
ODB solvers for test~3.  
The ODB solver reaches the tolerance $10^{-2}$ after 
$m=11$ steps, with no dimension reduction by truncated SVD.}
\label{f:mod3res}  
\end{figure}

\begin{figure}
\centerline{\includegraphics[width=0.6\linewidth]{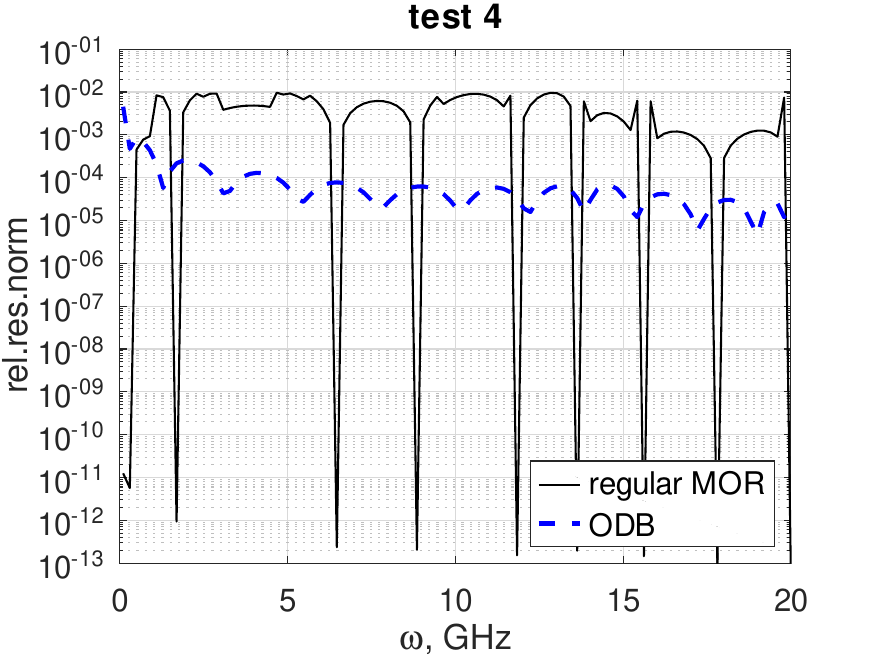}}
\caption{Relative residual norms versus frequency in the regular MOR and 
ODB solvers for test~4.
The ODB solver reaches the required tolerance $10^{-2}$ after 
$m=25$ steps. The ODB basis is then reduced to dimension~12 by SVD
truncated with tolerance $2\cdot 10^{-5}$.}
\label{f:mod6res}  
\end{figure}

\begin{figure}
\centerline{\includegraphics[width=0.6\linewidth]{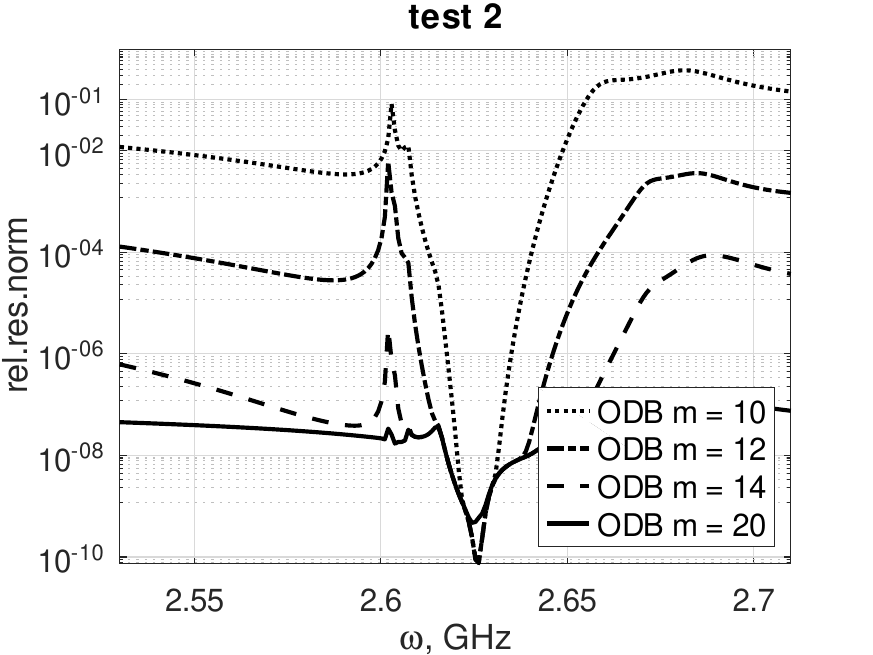}}
\caption{Relative residual norms for the ODB algorithms
versus frequency for different solution subspace dimensions~$m$, test~2}
\label{f:flt5res1}  
\end{figure}

\section{Conclusions and an outlook to future research}
\label{s:concl}
To solve a linear system parametric family, where the system
matrix~$A(k)$ is a quadratic pencil~$A(k)=A_0 + kA_1 + k^2A_2$ in the
parameter~$k$, we have proposed a new efficient
model order reduction (MOR) solution approach.  It is based on
two key ideas:\\
(1)~the MOR solution subspace is formed by an orthonormal basis of
$\text{span} (\bm{x}(\hat{k}), \bm{x}'(\hat{k}), \dots, \bm{x}^{(m)}(\hat{k}))$,
with $\bm{x}^{(m)}(\hat{k})$ being the $m$th order derivative of the
solution $\bm{x}(\hat{k})$;\\
(2)~the parameter value $\hat{k}$, where this orthonormal derivative basis
(ODB) is computed, is chosen in the complex plane outside the interval
$k\in[k_{\min},k_{\max}]\subset\Rr$, where the linear system family has
to be solved.\\
We have showed that the derivatives forming the basis can be computed
recursively, with just a single (sparse) LU~factorization of~$A(\hat{k})$,
and choosing the value~$\hat{k}$ in the complex plane properly gives
good approximation properties of the computed MOR basis.
The proposed approach, called ODB (orthonormal derivative basis)
method, is shown to be very efficient in the tests we have considered far,
reducing the computational costs drastically.

Moreover, for the problem class we consider, the system
matrices~$A(k)$, with $k\in[k_{\min},k_{\max}]\subset\Rr$, have spectral properties
which make solving the systems with~$A(k)$ by conventional
preconditioned Krylov subspace solvers hardly possible.
A pleasant side result of our research is
that spectral properties of the matrix~$A(\hat{k})$
for the chosen complex value~$\hat{k}$ appear to be much more favorable than
those of the matrix~$A(k)$ for $k\in[k_{\min},k_{\max}]\subset\Rr$
and allow to solve systems with~$A(\hat{k})$ by preconditioned
Krylov subspace solvers successively.  

The presented way to choose the parameter value $\hat{k}\in\Cc$
can be improved in a future research.  For instance, the
choice of $\hat{k}\in\Cc$ can be tuned to reach an even faster convergence
in Krylov subspace iterative solvers.  In addition, an interesting research
question is whether this approach (and, in particular, the idea to generate
the subspace for $\hat{k}$ outside the original interest set)
can be useful in further developing the Krylov subspace solver
recycling~\cite{Parks_deSturler_ea2006,Bolten_deSturler_ea2022}. 

\section*{Acknowledgments and declarations}
All the authors contributed equally to the paper.
The work of N.L.~Zamarashkin and D.A.~Zheltkov 
(general guidance, problem setting, introduction, highlights, 
derivation and implementation of the ODB method, conclusions) 
is supported by the Russian Science Foundation project no.~25-11-00392, 
\url{https://rscf.ru/project/25-11-00392/}.
The work of M.A.~Botchev 
(testing conventional MOR methods, making the plots, 
carrying out numerical experiments, writing the manuscript)
is supported by ongoing state research funding of the
Keldysh Institute of Applied Mathematics of Russian Academy of Sciences.
 
The authors declare that no neural network techniques have been used 
for producing results reported in this work and writing the paper.

\bibliographystyle{elsarticle-num}
\bibliography{parts/mor_pap}

\appendix
\section{Geometric convergence of the ODB method in a special case}
\label{a:Hermite} 

Here we consider a special case where $A(k)\equiv A_0+kA_1$ is a linear 
matrix pencil in $k>0$. We also assume that the matrices $A_0$ and $A_1$ are normal 
and commute each other. Then $A_0$ and $A_1$ have a common system of orthonormal eigenvectors $\bm{v}_j\in\Cc^N$, for which
\begin{equation}
\label{decA0A1}
A_0\bm{v}_j=\lambda_j\bm{v}_j, \qquad A_1\bm{v}_j=\mu_j\bm{v}_j,
\qquad \lambda_j,\mu_j\in\Cc.
\end{equation}
We also assume that $\mu_j\neq0$ and all the quotients $\lambda_j/\mu_j$ belong to 
a simply connected compact set in the complex plane not intersecting the negative 
real semi-axis. 
Introducing auxiliary vectors
\begin{equation}
\label{defuk}
\bm{u}(k)=\bm{x}(k)/k, 
\end{equation}
we rewrite our system~\eqref{mor} as $A(k)\bm{u}(k)=\bm{b}$.
For the right-hand side vector $\bm{b}$ written as a decomposition
\[
\bm{b}=\sum_j b_j\bm{v}_j, \qquad b_j\in\Cc,
\]
it is easily seen from \eqref{decA0A1} that
\[
\bm{u}(k)=\sum_j\frac{b_j}{\lambda_j+k\mu_j}\,\bm{v}_j
\]
and, by induction in $m$,
\[
\bm{u}^{(m)}(k)=(-1)^m m! \sum_j\frac{b_j\mu_j^m}{(\lambda_j+k\mu_j)^{m+1}}\,\bm{v}_j.
\]
Let us now choose a fixed value $k$ and try to approximate the vector $\bm{u}(l)$ 
by a linear combination of the vectors $\bm{u}^{(m)}(k)$, $m=0,1,\ldots M$, 
for some $l>0$ and $M$. This leads to an approximation problem
\[
\bm{u}(l) = \sum_j\frac{b_j}{\lambda_j+l\mu_j}\, \bm{v}_j \approx
\sum_j\frac{b_j}{\lambda_j+k\mu_j}
\cdot P\left(\frac{\mu_j}{\lambda_j+k\mu_j}\right) \bm{v}_j,
\]
with $P$ being an $M$th degree polynomial to be choosen. 
Since the eigenvectors $\bm{v}_j$ are orthonormal, this problem
can be uncoupled into a set of approximation problems
\[
\frac{1}{\lambda_j+ l\mu_j} \approx \frac{1}{\lambda_j+k\mu_j} \cdot
P\left(\frac{\mu_j}{\lambda_j+k\mu_j}\right),
\qquad j=1,\dots,N,
\]
or, equivalently, 
\[
P\left(\frac{1}{\lambda_j/\mu_j+k}\right)
\approx \frac{\lambda_j/\mu_j+k}{\lambda_j/\mu_j+l},
\qquad j=1,\dots,N.
\]
The quantities 
\[
\varkappa_j \equiv \frac{1}{\lambda_j/\mu_j+k}
\qquad \text{(thus,}\quad
\frac{\lambda_j}{\mu_j} = \frac{1}{\varkappa_j}-k
\text{)}
\]
belong to a compact set $C$ separated from the set
\[
D \equiv [-\infty,0)\cup[1/k,+\infty).
\]
Then, our approximation problem can be written in terms of $\varkappa_j$ as
\[
P(\varkappa_j) \approx \frac{1}{1-(k-l)\varkappa_j},
\]
or, in a somewhat stronger form, as
\[
P(\varkappa) \approx \frac{1}{1-(k-l)\varkappa}, \qquad \varkappa\in C.
\]
For any $l>0$ the pole $1/(k-l)$ lies in $D$. 
It follows from the theory of Faber series~\cite{Suetin} that for any 
$l>0$ the approximation error decreases geometrically 
with $M$ growing\footnote{Note that our convergence assertion 
can be applied to a class of pencils with arbitrary $N$, 
hence, we can allow $M$ to tend to infinity.}. 
The same, in turn, is valid for the approximation of $\bm{u}(l)$ 
by a linear combination of the vectors $\bm{u}^{(m)}(k)$, $m=0,1,\ldots, M$.

We now turn back to the vector $\bm{x}(k)$ and its derivatives in $k$. 
Since $\bm{x}(k)=k\bm{u}(k)$ (cf.~\eqref{defuk}), we have
\bea \label{recxy}
\bm{x}^{(0)}(k)=\bm{u}^{(0)}(k), \qquad
\bm{x}^{(m)}(k)=m\bm{u}^{(m-1)}(k)+k\bm{u}^{(m)}(k), \quad m\geqslant 1.
\eea
Recurrence \eqref{recxy} implies
\be \label{xyspans}
\mathop{\textrm{span}}\left\{\bm{x}^{(0)}(k),\ldots,\bm{x}^{(M)}(k)\right\}
=\mathop{\textrm{span}}\left\{\bm{u}^{(0)}(k),\ldots,\bm{u}^{(M)}(k)\right\}.
\ee
Since the vectors $\bm{x}(l)$ and $\bm{u}(l)$ are collinear and 
\[
\mathop{\textrm{dist}}\left(\bm{u}(l),
\mathop{\textrm{span}}\left\{\bm{u}^{(0)}(k),\ldots,\bm{u}^{(M)}(k)\right\}
\right) 
\rightarrow 0 \qquad\mbox{as}\ M\to\infty,
\]
equality (\ref{xyspans}) shows that
\[
\mathop{\textrm{dist}}\left(\bm{x}(l),\mathop{\textrm{span}}\left\{\bm{x}^{(0)}(k),\ldots,\bm{x}^{(M)}(k)\right\}\right) \rightarrow 0 
\qquad\mbox{as}\ M\to\infty.
\]

\section{Basis dimension reduction by truncated SVD}
\label{a:SVD}
Here we describe a reduction algorithm, which is based on truncated SVD
and replaces the computed basis matrix 
$V\in\Cc^{N\times m}$ by $V_{\mathrm{red}}\in\Cc^{N\times\hat{m}}$
with $\hat{m}<m$.

Recall that the approximate solution vectors $\bm{x}_j=\bm{x}(k_j)$ can be, for every $k_j$, $j=1,\dots,J$,
expressed as $\bm{x}_j=V y_j$, and the orthonormal columns of $V\in\Cc^{N\times m}$ form
the solution basis.  Let
$$
Y=
\begin{bmatrix}
y_1 & \dots & y_J
\end{bmatrix}\in\Cc^{m\times J}.
$$ 
We reduce the solution basis by truncating the SVD of the matrix $Y$.
Assume that $m\leqslant J$, which is a typical situation (otherwise the formulae below should be
applied to $Y^*$).  Then the full (not truncated) SVD of $Y$ is a factorization 
$$
Y=U_y \Sigma_y V_y^*, \quad
U_y\in\Cc^{m\times m}, \quad
\Sigma_y\in\Rr^{m\times m},\quad
V_y\in\Cc^{J\times m},
$$
where the matrices $U_y$ and $V_y$ have orthonormal columns and the matrix $\Sigma_y$
is diagonal with nonnegative ordered diagonal entries $\sigma_1\geqslant\dots\geqslant\sigma_m\geqslant 0$.
Let, for $\hat{m}\leqslant m$,
$$
\hat{\Sigma}_y = 
\begin{bmatrix}
\sigma_1      &     &        &   \\
              & \sigma_2 &        &   \\ 
              &     & \ddots &   \\
              &     &        & \sigma_{\hat{m}} \\
\end{bmatrix}\in\Rr^{\hat{m}\times \hat{m}},
\qquad
\widehat{\Sigma}_y = 
\begin{bmatrix}
\hat{\Sigma}_y &     &        &   \\
              & 0   &        &   \\ 
              &     & \ddots &   \\
              &     &        & 0 \\
\end{bmatrix}\in\Rr^{m\times m}.
$$
Consider an approximation $\hat{Y}$ to $Y$ defined as
$$
Y\approx \hat{Y} = \hat{U}_y\hat{\Sigma}_y\hat{V}_y^* = U_y\widehat{\Sigma}_yV_y^*,
$$
where the matrices $\hat{U}_y$ and $\hat{V}_y$ are formed by the first $\hat{m}$ columns
of the matrices $U_y$ and $V_y$, respectively.
Let $I_{\hat{m}}$ be the $\hat{m}\times \hat{m}$ identity matrix and let
$$
I_{\mathrm{trunc}} = \begin{bmatrix}
I_{\hat{m}}     &     &        &   \\
              & 0   &        &   \\ 
              &     & \ddots &   \\
              &     &        & 0 \\
\end{bmatrix}\in\Rr^{m\times m}.
$$
It is easy to see that the last $m-\hat{m}$ columns of the matrix $U_yI_{\mathrm{trunc}}$ are zero.
Hence, the last $m-\hat{m}$ columns of the matrix $VU_yI_{\mathrm{trunc}}$ are zero, too.
Let the matrix $V_{\mathrm{red}}\in\Cc^{N\times\hat{m}}$ be formed by the first $\hat{m}$ nonzero columns
of $VU_yI_{\mathrm{trunc}}$.
We can now see that the columns of $V_{\mathrm{red}}$ form a reduced solution basis: 
$$
VY \approx V\hat{Y}= V U_y\widehat{\Sigma}_yV_y^*=V U_yI_{\mathrm{trunc}}\widehat{\Sigma}_yV_y^*
=V_{\mathrm{red}}\hat{\Sigma}_y\hat{V}_y^*=V_{\mathrm{red}}Y_{\mathrm{red}},
$$
where we denote $Y_{\mathrm{red}}=\hat{\Sigma}_y\hat{V}_y^*\in\Cc^{\hat{m}\times J}$.
Thus, we obtain a new reduced solution basis:
\begin{equation}
\label{red_bas}
\bm{x}_j=Vy_j \quad \longrightarrow \quad \bm{x}_j\approx\hat{\bm{x}}_j= V_{\mathrm{red}} (y_{\mathrm{red}})_j,
\end{equation}
where $(y_{\mathrm{red}})_j$ is the $j$-th column of $Y_{\mathrm{red}}$.
This basis reduction introduces an approximation error $\bm{x}_j\approx\hat{\bm{x}}_j$.
The 2-norm of the error can easily be estimated.  Indeed,
\begin{equation}
\label{err_est}
\begin{aligned}
\|\bm{x}_j-\hat{\bm{x}}_j\| &= \|Vy_j - V\hat{y}_j\| = 
\|VU_y\Sigma_y (V_y^*)_j - VU_y\widehat{\Sigma}_y (V_y^*)_j\|=
\\
&=\|\Sigma_y (V_y^*)_j - \widehat{\Sigma}_y (V_y^*)_j\|
=\|(\Sigma_y - \widehat{\Sigma}_y) (V_y^*)_j\| =
\\
&=\left\|
\begin{bmatrix}
0 & \dots & 0 \\
  & \dots &  \\
0 & \dots & 0 \\
- & \sigma_{\hat{m}+1}v_{\hat{m}+1}^* & -\\
  & \dots    &  \\ 
- & \sigma_{m}v_{m}^* & -
\end{bmatrix}
\right\| \leqslant \sigma_{\hat{m}+1},
\end{aligned}
\end{equation}
where $(V_y^*)_j$ is the $j$-th column of $V_y^*$.
In practical computations, when carrying out this basis reduction,
we first normalize the vectors $y_j$, i.e., $y_j:=y_j/l_j$, $l_j=\|y_j\|$,
so that $\|\bm{x}_j\|=\|Vy_j\|=\|y_j\|=1$ (we use the 2-norm here),
and store the solution norms $l_j$.  Therefore, the error bound~\eqref{err_est}
is actually an estimate for the relative solution error $\|\bm{x}_j-\hat{\bm{x}}_j\|/\|\bm{x}_j\|=\|\bm{x}_j-\hat{\bm{x}}_j\|$
as well.  Once the basis reduction is done, we re-normalize the vectors
$y_j:=l_jy_j$.

\end{document}